\documentclass[12pt]{article}

\usepackage{preprint-layout} % estetics
\usepackage{preprint-notation} % standard notation
\usepackage{algpseudocode}

\usepackage{hyperref}
\usepackage{underscore} % allow for underscore in doi

\DeclareMathOperator{\supp}{supp}
\DeclareMathOperator{\diam}{diam}
\DeclareMathOperator{\Span}{span}

\newcommand{\mcThs}{\mcT_h^{{S}}} % small elements 
\newcommand{\mcThl}{\mcT_h^{{L}}}    % large elements
\newcommand{\Omegahl}{\Omega_h^{{L}}}
\newcommand{\Omegah}{\Omega_h}

\newcommand{\hatA}{\widehat A}
\newcommand{\hatE}{\widehat E}

\newcommand{\IQ}{\mathbb{Q}}
\newcommand{\IP}{\mathbb{P}}
\title{Extension Operators for\\Trimmed Spline Spaces} 

\author{Erik Burman,  Peter Hansbo, Mats G. Larson, Karl Larsson}

\date{\today}

\begin{document}

\maketitle

\begin{abstract}
We develop a discrete extension operator for trimmed spline spaces consisting of piecewise polynomial functions of degree $p$ with $k$ continuous derivatives. The construction is based on polynomial extension from neighboring elements together with projection back into the spline space. We prove stability and approximation results for the extension operator. Finally, we illustrate how we can use the extension operator to construct a stable cut isogeometric method for an elliptic model problem.
\end{abstract}

\section{Introduction} 
\paragraph{Contributions.}
Consider a function space of $C^k$ piecewise polynomial splines defined on a background mesh. Let $\Omega$ be a domain embedded into the background mesh without requiring that boundary matches the mesh leading to so-called cut or trimmed elements in the vicinity of the boundary. The span of every basis function that has a support that intersects $\Omega$ form the so-called active spline space associated with $\Omega$. Due to the presence of the cut elements, this basis is typically very ill-conditioned since the intersection of the domain and the support of some basis functions may be very small. We, therefore, introduce an extension operator that expresses the ill-conditioned degrees of freedom in terms of the well-conditioned degrees of freedom in the interior of the domain in a stable and accurate way.  Such a discrete extension operator has many uses, for instance,  handling of trimmed elements \cite{MaHu18} in isogeometric analysis \cite{igabook09}  in a robust manner and as an alternative to adding stabilization terms to the weak form in cut and immersed finite elements,  see \cite{Bu10, MLLR14, BCHLM15}.
%\todo{add ref trimmed iga}

To keep the presentation as simple as possible we consider $C^k$ tensor product splines but the construction is very general and allows local refinement as well as splines on triangulations. The basic idea is to split the mesh into elements with a large respectively a small intersection with the domain $\Omega$. To each small element we associate a large element in a neighborhood with a size proportional to the local mesh parameter and we then use the polynomial extension from the associated large element to the small element. This preliminary extension manufactures a function, which may be discontinuous at the faces belonging to small elements. We then project back to the spline space using an interpolation operator. This 
construction enables us to handle very general situations and to prove that the extension is stable and has optimal approximation properties in a systematic way. 
  
\paragraph{Previous Work.} Extension operators for discontinuous spaces were constructed and analyzed in \cite{JohLar13}, for continuous spaces an early approach using cell merging in structured meshes was suggested in \cite{HWX17} and then a more general agglomeration approach was introduced in \cite{BadVer18}. A framework for all nodal finite element spaces, including for instance the family of Hermite splines,  was presented in \cite{BurHanLar21a}.  Examples of applications of extension operators to cut finite element methods include \cite{BMV18b} and \cite{BHL20b}.  Another approach using extension operators, where certain terms in the finite element formulation are evaluated at an extended polynomial, was proposed for Lagrange multiplier methods in \cite{HR09} and in the context of splines using Nitsche type weak imposition of boundary conditions in \cite{BPV20}. Trimmed elements and immersed methods attract significant interest in isogeometric analysis, see for instance \cite{HVABR18, MaZe17, MaHu18, CasBon15, CasZha18}.

\paragraph{Outline.}
The paper is organized as follows: In Section 2 we introduce the spline spaces including assumptions on 
the basis functions, in Section 3 we construct an interpolation operator and establish stability and approximation properties, in Section 4 we develop the extension operator and establish its stability and approximation properties,
in Section 5 we apply the extension to a cut isogeometric method for an elliptic model problem,
and in Section 6 we present some numerical examples based on that method.
  
\section{Spline Spaces}
\begin{itemize}
\item Let $\widetilde{\mcT}_h$ be a uniform tensor product mesh on $\IR^d$ with mesh 
parameter $h \in (0,h_0]$.  Let $\widetilde{V}_{h,p,k}\subset C^{k}(\IR^d)$ 
be a spline space of piecewise tensor product polynomials of order $p$ with regularity parameter
 $0 \leq k \leq p-1$. We also let $\widetilde{V}_{h,p,-1}$ denote the space of discontinuous piecewise tensor product polynomials of order $p$.

\item Let 
\begin{equation}
\widetilde{\mcB} = \{\widetilde{\varphi_i}\}_{i \in \widetilde{I}}
\end{equation}
denote a basis in $\widetilde{V}_{h,p,k}$ and define the sub mesh consisting of elements contained 
in the support of $\widetilde{\varphi_i}$,
\begin{equation}\label{eq:Thi}
\widetilde{\mcT}_{h,i} = \{ T \in \widetilde{\mcT}_{h} :   T \subset \supp( \widetilde{\varphi}_i ) \}
\end{equation}
and the set of indices $\widetilde{I}_T \subset \widetilde{I}$ of the basis functions that contain 
element $T$,
\begin{align}\label{eq:IT}
\widetilde{I}_T = \{ i \in I : T \subset \supp(\widetilde{\varphi}_i) \} 
\end{align}
and the corresponding basis functions 
\begin{equation}
\widetilde{\mcB}_T = \{ \widetilde{\varphi}_i : i \in \widetilde{I}_T \}
\end{equation}

\item Assume that the basis $\widetilde{\mcB}$ satisfies:
\begin{description}
\item[A1.] The basis functions are locally supported. There is a constant, such that 
\begin{equation}\label{eq:locsupp-a}
| \widetilde{\mcT}_{h,i} | \lesssim 1
\end{equation}

\item[A2.] The restrictions of the basis functions $\widetilde{\mcB}_T$  whose support contain $T$ is a basis for 
$\IQ_p(T)$, the tensor product of one dimensional polynomials of degree $p$ on $T$,
\begin{equation}
\IQ_p(T) = \text{span}\{ \varphi_j : j \in \widetilde{I}_T \}
\end{equation}
and the basis $\widetilde{\mcB}_T$ is stable in the sense that there are constants such that 
\begin{equation}\label{eq:BTstab}
h^d \|\hatv_T\|^2_{\IR^{N_T}}  \sim \| v \|^2_T, \qquad v \in \IQ_p(T) 
\end{equation}
where $N_T = \dim(\IQ_p(T))$ and 
\begin{equation}\label{eq:BTexp}
v = \sum_{i \in \widetilde{I}_T} \hatv_{T,i} \varphi_i|_T
\end{equation}
is the expansion of $v\in \IQ_p(T)$ in the basis $\widetilde{\mcB}_T$.
\end{description}
Here and below the $a\lesssim b$ means $ a \leq C b$ where the constant $C$ is independent of $h$ and the intersection of the computational mesh and the domain boundary, 
but may depend on $p$ and $k$.

\begin{rem} The number of elements in $\mcT_{h,i}$ for the standard basis of B-splines with maximum regularity $k=p-1$ is $(p+1)^d$.  
\end{rem}
\begin{rem} Our constructions and analysis extend to more general spline spaces including local refinements,  
splines on triangulations,  tensor products of splines of various order and regularity.  The essential assumption is 
that we have a piecewise polynomial space and a basis that is local and the restriction of the basis functions to 
an element spans a suitable polynomial space.  We have chosen the most 
common situation of uniform tensor product splines to keep the presentation simple. 
\end{rem}

\item Given a domain $\Omega \subset \IR^d$ with Lipschitz boundary $\partial \Omega$ we define the active mesh 
\begin{equation}
\mcT_h = \{T \in \widetilde{\mcT}_h : T \cap \Omega \neq \emptyset \}
\end{equation}
and we let 
\begin{equation}
\Omega_h = \cup_{T \in \mcT_h} T
\end{equation}
We let the active basis $\mcB \subset \widetilde{\mcB}$ consist of all basis functions that contain an active element in their support,
\begin{equation}
\mcB = \{ \varphi \in \widetilde{\mcB} : \exists T \in \mcTh, T \subset \supp(\varphi) \} = \cup_{T \in \mcTh} \widetilde{\mcB}_T
\end{equation}
and we denote the index set of $\mcB$ by $I$.  The associated active spline space $V_{h,p,k}$ is defined by 
\begin{equation}
V_{h,p,k} = \text{span} (\mcB)
\end{equation}
For $T \in \mcTh$ we use the notation $ \widetilde{\mcB}_T = \mcB_T $,  and 
$\widetilde{I}_T = I_T$.
\end{itemize}

\section{Interpolation}
 In this section, we construct an interpolation operator $\pi_h:L^2(\Omega) \rightarrow V_{h,p,k}$ based 
 on standard interpolation operators for spline spaces composed with a continuous extension of the 
 function outside of the domain. We establish basic stability and approximation properties.  We also 
 show an estimate of the error in the interpolant of a discontinuous piecewise polynomial function, where the bound is in terms of a jump operator measuring the size of the jumps in derivatives across faces.  For background on spline (quasi-)interpolants we refer to \cite{deBoFi73,LeeLyc01,LycSch75,BeiBuf14}.

\subsection{Definitions}
\begin{itemize}
\item There is an extension operator  
\begin{equation}\label{eq:ext}
E: H^s(\Omega) \ni v \mapsto v^E \in H^s(\IR^d)
\end{equation}
independent of $s$,  such that 
\begin{equation}\label{eq:ext-stab}
\|v^E\|_{H^s(\IR^d)} \lesssim \|v\|_{H^s(\Omega)}
\end{equation}
see \cite{stein70}. When not needed for clarity we simply write $v^E = v$.

\item  For each element $T \in \mcT_h$ let $P_{T,p}: L^2(T) \rightarrow \mathbb{Q}_p(T)$ 
be the $L^2(T)$ projection. For $v\in H^s(\Omega)$ and each element $T\in \mcTh$ we may 
expand the projection $P_{T,p} (v^E|_T) \in \IQ_p(T)$ in the basis $\mcB_T$, 
\begin{equation}
   P_{T,p} v|_T = \sum_{i \in I_T} \widehat{v}_{T,i} \varphi_i|_T 
\end{equation}
Thus for each element $T \in \mcT_{h,i}$ we obtain a potential coefficient $\hatv_{T,i}$ 
multiplying basis function $\varphi_i$. We finally, take an average over the coefficients $\hatv_{T,i}$ obtained from each 
of the elements $T$ in the support of $\varphi_i$ to get the final coefficient for $\varphi_i$.  More 
precisely, we define 
\begin{align}\label{def:pih}
\boxed{\pi_h v = \sum_{i \in I} \langle \hatv_{T,i} \rangle_{T \in \mcT_{h,i}} \varphi_i }
\end{align}
where the average is a convex combination
\begin{equation}\label{def:pihcoeff}
\boxed{\langle \hatv_{T,i} \rangle_{T \in \mcT_{h,i}} = \sum_{T \in \mcT_{h,i}} \kappa_{T,i} \hatv_{T,i} }
\end{equation}
with arbitrary weights $\kappa_{T,i}$ such that $0\leq \kappa_{T,i} \leq 1$ and 
$\sum_{T \in \mcT_{h,i}} \kappa_{T,i} = 1$.  Note that the weights can be individually chosen 
for each basis function.
\end{itemize}

\begin{rem} Let $\mcB_T^* = \{ \varphi_{T,j}^* :  j \in I_T\}$  be the dual basis to $\mcB_T$ on $T$,  
characterized by 
\begin{equation}
 \varphi_{T,j}^*(\varphi_i|_T) = \delta_{ij}
\end{equation}
Then there is $\chi_{T,i}^* \in \IQ_p(T)$ such that $ \varphi_{T,j}^*(w) = (\chi_{T,i}^*,w)_T$ for $w \in \IQ_p(T)$ 
and we can extend the action of $\varphi_{i,T}^*$ to $L^2(T)$.  We then have 
$ \varphi_{T,j}^*(P_{T,p} v|_T) =  \varphi_{T,j}^*(v|_T)$ and therefore 
\begin{align}
  \varphi_{T,j}^*(v|_T) =  \varphi_{T,j}^* (P_{T,p} v|_T)= \sum_{i \in I_T} \widehat{v}_{T,i}\varphi_{T,j}^* ( \varphi_i|_T ) = \widehat{v}_{T,j}
\end{align}
which allow us to express the interpolant in terms of the dual basis 
\begin{align}
\pi_h v = \sum_{i \in I} \langle \varphi^*_{T,i}(v|_T) \rangle_{T \in \mcT_{h,i}} \varphi_i 
\end{align}
\end{rem}

\subsection{Properties}

The spline space $V_{h,p,k}$ is invariant under the interpolation operator. To see 
this we consider an element $T$, clearly $P_{T,p} v = v$ on $T$. Expanding $v$ in the 
spline basis $\mcB$ and restricting to $T$, 
\begin{align}
  \sum_{i \in I} \hatv_i \varphi_i|_T = v|_T = P_{T,p} v = \sum_{i \in I_T} \widehat{v}_{T,i} \varphi_i|_T 
\end{align}
Since the expansions are unique we conclude that $\widehat{v}_{T,i} = \hatv_i$, for all $T \in \mcT_{h,i}$, 
and therefore $\langle \hatv_{T,i} \rangle_{T \in \mcT_{h,i}} = \hatv_i$. It follows that $\pi_h v = v$ 
for $v \in V_{h,p,k}$. Furthermore, introducing the notation 
\begin{equation}\label{eq:omegahT}
%\mcN_{h,T} = \cup_{i \in I_T} \mcT_{h,i}, \qquad 
%\omega_{h,T} = \cup_{T \in \mcN_{h,T}} 
\omega_{h,T} = \cup_{i \in I_T} \supp(\varphi_i)
\end{equation}
we note that the restriction of the interpolant $\pi_h v$ to $T$ depends only on $v$ restricted 
to $\omega_{h,T}$, and since $\diam(\supp (\varphi_i))\lesssim h$ we have
\begin{equation}\label{eq:omegahTdiam}
\diam(\omega_{h,T} ) \lesssim h
\end{equation}
We now proceed with some basic stability and approximation results for the interpolation operator.

\begin{lem}\label{lem:pih-stab} There is a constant such that 
\begin{align}\label{eq:pihstab}
\boxed{\| \pi_h v \|_{T} \lesssim \| v \|_{\omega_{h,T}}, \qquad v \in L^2(\omega_{h,T})}
\end{align}
\end{lem}
\begin{proof} Starting from the definition (\ref{def:pih}) of the interpolant and the expression for the 
coefficients (\ref{def:pihcoeff}) we obtain 
\begin{align}
\| \pi_h v \|^2_{T} &\lesssim \sum_{i \in I_T} | \langle \hatv_{T,i} \rangle_{T \in \mcT_{h,i}}|^2 h^d
 \lesssim \sum_{i \in I_T}  \sum_{T' \in \mcT_{h,i}} \hatv_{T',i}^2  h^d
\\
&\qquad
\lesssim  \sum_{i \in I_T}  \sum_{T' \in \mcT_{h,i}} \| P_{T',p} (v^E|_T') \|^2_{T'} 
\lesssim  \sum_{i \in I_T}  \sum_{T' \in \mcT_{h,i}} \| v^E \|^2_{T'}
\lesssim \| v^E \|^2_{\omega_{h,T}}
\end{align}
where we used the equivalence (\ref{eq:BTstab}). This completes the proof.
\end{proof}

\begin{lem}\label{lem:pih-approx} There is a constant such that for all $v \in H^r(\Omega)$
\begin{align}\label{eq:pihapprox}
\boxed{\| v - \pi_h v \|_{H^m(\Omegah)} \lesssim h^{s - m} \| v \|_{H^{s}(\Omega)},
\qquad  0\leq m \leq k}
\end{align}
where $s = \min(r,p+1)$.
\end{lem}
\begin{proof} We have the invariance
\begin{equation}
(\pi_h w)|_T = w , \qquad w \in \IP_p(\omega_T)
\end{equation}
since the spline space is invariant under the action of $\pi_h$ 
and polynomials can be represented exactly in the spline space.
Therefore, for any $w \in \mathbb{P}_p(\omega_{h,T})$ we have,
\begin{align}
\| v - \pi_h v \|_{H^m(T)} &\leq \| v - w \|_{H^m(T)} + \| w - \pi_h v \|_{H^m(T)}
\\
&\lesssim \| v - w \|_{H^m(\omega_{h,T})} + h^{-m} \| w - \pi_h v \|_T
\\
&\lesssim \| v - w \|_{H^m(\omega_{h,T})} + h^{-m} \| \pi_h (w -  v) \|_T
\\
&\lesssim \| v - w \|_{H^m(\omega_{h,T})} + h^{-m} \| w -  v \|_{\omega_{h,T}}
\\
&\lesssim h^{s - m}\| v \|_{H^{s}(\omega_{h,T})}
\end{align}
where we first used the stability (\ref{eq:pihstab}) and then choose $w$ according to the 
Bramble-Hilbert lemma, see \cite[Lemma 4.3.8]{BreSco}.  Summing over the elements and using the 
stability (\ref{eq:ext-stab}) of the extension operator completes the proof.
\end{proof}

\begin{lem}\label{lem:pih-approx-disc}
There is a constant such that  
\begin{align}\label{eq:pih-approx-disc}
\boxed{\| v - \pi_h v \|_{\Omega_h} \lesssim \| v \|_{j_h} , \qquad v \in V_{h,p,-1}}
\end{align}
with
\begin{equation}\label{eq:jh}
\| v \|^2_{j_h} = \sum_{l=0}^p h^{2l + 1} \| [ \nabla^l v ] \|^2_{\mcF_h}
\end{equation}
where $\mcF_{h}$ is the set of interior faces in $\mcT_{h}$.
\end{lem}
\begin{proof} 
Proceeding in a similar manner as in the proof of Lemma \ref{lem:pih-approx} we note that for 
any $w \in \mathbb{P}_p(\omega_T)$ we have, using the $L^2$ stability of $\pi_h$,
\begin{align}
\| (I - \pi_h )  v \|_T &= \| (I - \pi_h )  (v - w) \|_T \lesssim \| v - w\|_{\omega_{h,T}}
\end{align}
Finally, taking $w\in \mathbb{P}_p(\omega_T)$ such that $w = v$ on $T$ and using standard 
estimates, see \cite{HanLar17} and \cite{MasLar14},  for face stabilisation of higher order elements we get 
\begin{align}
\| v - w \|^2_{\omega_{h,T}} \lesssim \sum_{l=0}^p h^{2l+1} \| [ \nabla^l v ] \|^2_{\mcF_{h,i}}
\end{align} 
where $\mcF_{h,i}$ is the set of interior faces in $\mcT_{h,i}$.
\end{proof}

\begin{rem} Note that the $h$-scaling in $j_h$ is chosen such that we have the inverse inequality
\begin{equation}
\| v \|_{j_h} \lesssim \| v \|_{\mcT_h} , \qquad v \in V_{h,p,-1} \cap C(\Omega_h)
\end{equation} 
\end{rem}

\begin{rem} The estimate (\ref{eq:pih-approx-disc}) may be 
 viewed as a generalization of estimates for Oswald interpolation introduced in \cite{Osw93},  from piecewise linear spaces to spline spaces. This type of operator is also used for the analysis of interior penalty stabilised finite element methods \cite{DD76,BH04}. In the case of high order finite element spaces an $hp$-analysis was considered in \cite{BE07}. In this context the Lemma \ref{lem:pih-approx-disc} is instrumental for the analysis of the skeleton based stabilised methods proposed in \cite{HVABR18}.
\end{rem}

\section{Extension}
\subsection{Definitions}
\begin{itemize} 
\item To define the extension operator we partition $\mcT_h$ into the set of elements $\mcThl$ 
that have a large intersection with $\Omega$, in the sense that 
\begin{equation}\label{eq:Largeelem}
\gamma h^d \leq |T \cap \Omega | 
\end{equation}
for a parameter $\gamma\geq 0$, and the set of elements $\mcThs =\mcT_{h} \setminus \mcThl $ 
with a small intersection. We thus have 
\begin{equation}
\mcTh = \mcThl \cup \mcThs
\end{equation}
We also define 
\begin{equation}
\Omegahl = \cup_{T \in \mcThl} T \subset \Omegah
\end{equation}

\item Let $S_h:\mcThs \rightarrow \mcThl$ be a mapping that associates a 
large element to a small element such that 
\begin{align}\label{eq:shassump}
\diam (S_h(T) \cup T) \lesssim h 
\end{align}
According to Lemma 2.4 in \cite{BurHanLar21a} there is such a mapping for domains with Lipschitz boundary when the mesh is sufficiently fine.

\item We let 
\begin{equation}
\mcB^L = \cup_{T \in \mcThl} \mcB_T \subset \mcB, \qquad \mcB^S = \mcB \setminus \mcB^L
\end{equation}
be the set of active basis functions that contain a large element in their support and the set of 
remaining active basis functions.  We then have the (interior) direct sum
\begin{align}
\boxed{V_{h,p,k} = V_{h,p,k}^L \oplus V_{h,p,k}^S}
\end{align}
where 
\begin{equation}
 V_{h,p,k}^L = \Span(\mcB^L), \qquad V_{h,p,k}^S = \Span(\mcB^S)
\end{equation}
The index sets of $\mcB^L$ and $\mcB^S$ are denoted by $I^L$ and $I^S$, respectively.

\item We define the extension operator 
\begin{equation}
B_h: V^L_{h,p,k}  \rightarrow V_{h,p,-1}
\end{equation} 
in such a way that
\begin{align} \label{def:Bh}
(B_h v)|_T = 
\begin{cases}
\left(v|_{S_h(T)}\right)^e|_T, & T \in \mcThs 
\\
v|_T, & T \in \mcThl 
\end{cases}
\end{align}
where $v^e$ denotes the canonical extension of a polynomial $v$ on $T$ to a polynomial on $\IR^d$.

\item We define the extension operator 
\begin{equation}\label{def:Eh}
\boxed{E_h: V^L_{h,p,k}  \ni v \mapsto \pi_h B_h v  \in V_{h,p,k}^E \subset V_{h,p,k} }
\end{equation}
where the extended finite element space  $V_{h,p,k}^E$ is defined by 
\begin{equation}
\boxed{ V_{h,p,k}^E = E_h V^L_{h,p,k} }
\end{equation}

\item Extending the mapping $S_h$ from $\mcThs$ to $\mcTh$ by setting $S_h(T) = T$, for 
$T \in \mcThl$, we note that the set valued mapping $S_h^{-1}: \mcThl \rightarrow \mcTh$ induce a partition 
$\mcM_h$ of $\mcT_h$ into macro elements 
\begin{equation}\label{eq:macro}
M_T = \cup_{T' \in S_h^{-1}(T)} T', \qquad T \in \mcThl
\end{equation}
that thanks to the property (\ref{eq:shassump}) satisfy 
\begin{equation}\label{eq:macrodiam}
\diam(M_T) \lesssim h
\end{equation}
 The macro elements $M_T$ are invariant under $S_h$ in the sense that if $T' \subset M_T$ then $S_h(T') \subset M_T$. Let $V_{h,p,-1}^M$ be the space of discontinuous tensor product polynomials polynomials of degree $p$ on $\mcM_h$. Then it follows from the invariance of $\mcM_h$ under $S_h$ that $V_{h,p,-1}^M$ is invariant under $B_h$, 
\begin{equation}\label{eq:Bh-inv}
\boxed{w = B_h w, \qquad w \in V^M_{h,p,-1}}
\end{equation}

\item If we choose the weights $\kappa_{T,i}$ in the average (\ref{def:pihcoeff}) such that 
\begin{equation}\label{eq:weights-rest}
\kappa_{T',i} = 0, \qquad T' \in \mcThs
\end{equation}
for all basis functions $\varphi_{i} \in \mcB_T^L$.  Then the extension operator takes the form
\begin{equation}
\boxed{ E_h: V_{h,p,k}^L \ni  v^L \mapsto v^L \oplus (E_h v^L)^S \in V_{h,p,k}^E \subset V_{h,p,k} }
\end{equation}
where $v^L = \sum_{i\in I^L} \widehat v_i \varphi_i \in V_{h,p,k}^L$ denotes the expansion of $v\in V_{h,p,k}$ in the large basis functions and  $v^S = \sum_{i\in I^S} \widehat v_i \varphi_i \in V_{h,p,k}^S$ denotes the expansion of $v$ in the small basis functions.
This means that the component $(E_h v)^L$ of $E_h v$ in $V^L_{h,p,k}$ is identical to $v^L$ 
and the extension operator determines a suitable component in $V^S_{h,p,k}$. In particular, we note that with weights satisfying (\ref{eq:weights-rest}) the extension operator does not change 
$v$ on $\Omegahl$,
\begin{equation}\label{eq:invariant-large}
\boxed{(E_h v) |_{\Omegahl} = v|_{\Omegahl}}
\end{equation}
\end{itemize}

\subsection{Properties}

\begin{lem}[Preservation of Polynomials] If $v \in \IQ_p(\Omega)$ then 
\begin{equation}
\boxed{ E_h \bigl( v |_{\Omegahl} \bigr) = v }
\end{equation}
\begin{proof} It follows from the definition of $B_h$ that $B_h \bigl(v|_{\Omegahl}\bigr) = v$ and $\pi_h v = v$.
\end{proof}

\end{lem}
Before proving a stability result for the extension operator we show the following technical lemma which 
provides a bound for the right hand side of (\ref{eq:pih-approx-disc}) for a function of the form 
$B_h v \in V_{h,p, -1}$ with $v \in V_{h,p,k}$.
\begin{lem}\label{lem:jumpest}  There is a constant such that 
\begin{align}
\boxed{ \| B_h v \|_{j_h} \lesssim h^m \| \nabla^m v \|_{\Omega},\qquad 0\leq m \leq k, \qquad v \in V_{h,p,k} }
\end{align}
\end{lem}
\begin{proof} Consider a face $F$ shared by elements $\mcT_h(F) = \{T_1,T_2\}$. Let 
$\omega_\delta$ be a ball of radius $\delta$ such that 
\begin{align}
T_1 \cup S_h (T_1) \cup T_2 \cup S_h(T_2) \subset \omega_\delta
\end{align}
Since $T_1$ and $T_2$ share a face and $S_h$ satisfies (\ref{eq:shassump}) we conclude that there is such a ball with radius $\delta \lesssim h$. 
For $w \in \IP_p (\omega_\delta)$ we then have the estimates 
\begin{align}
&h^{2l + 1} \| [ \nabla^l B_h v ] \|^2_{F} = h^{2l + 1} \| [ \nabla^l (B_h v - w )] \|^2_{F} 
\lesssim  h^{2l} \| \nabla^l ( B_h v - w) \|^2_{\mcT_h(F)}  
\\
&\quad  \lesssim \| B_h v - w \|^2_{\mcT_h(F)}  
\lesssim \| v - w \|^2_{S_h(\mcT_h(F))}
  \lesssim \| v^E - w\|^2_{\omega_\delta}
  \lesssim h^{2m} \| v^E \|^2_{H^m(\omega_\delta)}
\end{align}
where we used inverse inequalities to pass from the face to elements and to remove derivatives, then we 
used the identity 
\begin{equation}
 (B_h v - w)|_{T_i}  = (v|_{S_h(T_i)})^e -   (w|_{S_h(T_i)})^e =  ((v-w)|_{S_h(T_i)})^e 
\end{equation}
followed by stability 
\begin{equation}\label{eq:polextstab}
\| (q|_{S_h(T)})^e  \|_{T} \lesssim \| q \|_{S_h(T)}, \qquad q \in \mathbb{Q}_p(S_h(T)) 
\end{equation}
of polynomial extension, and finally, in the last inequality we choose $w$ to be the $L^2(\omega_\delta)$ projection of the continuous extension $v^E$ restricted to $\omega_\delta$ and 
used a standard approximation result on the ball with diameter $\delta \lesssim h$. Note that we need the continuous extension in the last step since the ball $\omega_\delta$ may not be contained in $\Omega$.
Summing over all faces and using the stability (\ref{eq:ext-stab}) of the continuous extension operator the 
desired estimate follows.
\end{proof}

\begin{lem}[Stability]\label{lem:exth-stab} There is a constant such that
\begin{align}
\boxed{ \| \nabla^m E_h v \|_{\Omegah} \lesssim \| \nabla^m v \|_{\Omega} , \qquad 0\leq m \leq k, 
\quad v \in V_{h,p,k} }
\end{align}
\end{lem}
\begin{proof}  Using the stability (\ref{eq:polextstab}) of polynomial extension we get  
\begin{equation}
\| \nabla^m v^e \|_{M_T} \lesssim \| \nabla^m v \|_{T}, \qquad v \in \IQ_p(T), \quad T \in \mcT_h^L
\end{equation}
where we recall that $M_T = \cup_{T \in S_h^{-1}( T )} T$,  and therefore 
\begin{equation}\label{eq:Bhstab}
\| \nabla^m B_h v \|_{\mcM_h} \lesssim \| \nabla^m v \|_{\mcThl}
\end{equation}
Adding and subtracting $v \in V_{h,p,k}$ we have 
\begin{align}
\| \nabla^m E_h v \|_{\Omegah} &= \|\nabla^m \pi_h B_h  v \|_{\Omegah} 
\\
&\lesssim \| \nabla^m (\pi_h - I) B_h  v \|_{\mcTh} + \| \nabla^m B_h v \|_{\mcM_h}
\\
&\lesssim  h^{-m} \| (\pi_h - I) B_h  v \|_{\mcTh} + \| \nabla^m v \|_{\mcThl}
\end{align}
where we used the stability (\ref{eq:Bhstab}) of $B_h$ to estimate the second term. For 
the first term we employ  Lemma \ref{lem:pih-approx-disc} and Lemma \ref{lem:jumpest}, 
\begin{align}
h^{-m} \| (\pi_h - I) B_h  v \|_{\mcTh} & \lesssim h^{-m} \| B_h v \|_{j_h} \lesssim \| v \|_{H^m(\Omega)} 
\end{align}
which completes the proof.
\end{proof}

We now define an interpolant $\pi_h^E:L^2(\Omega) \rightarrow V_{h,p,k}$ by applying the discrete extension operator to the interpolation operator defined in \eqref{def:pih}, $\pi_h^E v := E_h \pi_h v$. This interpolation into the extended space satisfies the following approximation result.
\begin{lem}[Approximation] \label{lem:exth-approx} There is a constant such that for all $v \in H^r(\Omega)$,
\begin{align}
\boxed{ \| v - \pi_h^E v \|_{H^m(\Omegah)} \lesssim h^{s - m} \| v \|_{H^s(\Omega)}, \quad 0 \leq m \leq k, \quad s = \min(r,p+1) }
\end{align}
\end{lem}
\begin{proof}  Let $\pi_h^M:L^2(\Omegah) \rightarrow V_{h,p,-1}^M$ be the element wise 
$L^2$ projection. We then have 
\begin{align}\label{eq:macro-approx}
\| v - \pi_h^M v \|_{\mcM_h} \lesssim h^s | v |_{H^s(\mcM_h)} , \quad \quad s = \min(r,p+1)
\end{align}
Adding and subtracting  $\pi_h v$ and $\pi_h \pi_h^M v = \pi_h B_h \pi_h^M v$, where 
we used the invariance (\ref{eq:Bh-inv}), we get the identity 
\begin{align}
 v - \pi_h B_h (\pi_h v) &=  v - \pi_h v + \pi_h v  - \pi_h \pi_h^M v  + \pi_h \pi_h^M v  - \pi_h B_h (\pi_h v)
 \\
  &=  v - \pi_h v + \pi_h (v  -  \pi_h^M v)  + \pi_h B_h (\pi_h^M v  - \pi_h v)
\end{align}
Using the triangle inequality, inverse inequalities, the $L^2$ stability of $\pi_h$ and $B_h$,
\begin{align}
 &\|v - \pi_h B_h (\pi_h v)\|_{H^m(\mcTh)} 
\\
 &\qquad \leq \| v - \pi_h v \|_{H^m(\mcTh)} + \|\pi_h (v  -  \pi_h^M v)  \|_{H^m(\mcTh)}  
  + \| \pi_h B_h (\pi_h^M v  - \pi_h v) \|_{H^m(\mcTh)} 
  \\
    &\qquad \lesssim \| v - \pi_h v \|_{H^m(\mcTh)} + h^{-m} \|\pi_h (v  -  \pi_h^M v)  \|_{\mcTh} 
  + h^{-m} \| \pi_h B_h (\pi_h^M v  - \pi_h v) \|_{\mcTh} 
    \\ \label{eq:exth-approx-b}
    &\qquad \lesssim \| v - \pi_h v \|_{H^m(\mcTh)} + h^{-m} \|v  -  \pi_h^M v \|_{\mcTh} 
  + h^{-m} \| \pi_h^M v  - \pi_h v \|_{\mcThl} 
\\
  &\qquad \lesssim h^{s - m } \| v \|_{H^{s}(\Omega)}
\end{align}
Here we used the interpolation estimate for $\pi_h$  in Lemma \ref{lem:pih-approx} and 
the estimate (\ref{eq:macro-approx}) for $\pi_h^M$. To estimate the third term in (\ref{eq:exth-approx-b}) 
we added and subtracted $v$ and once again used the approximation of $\pi_h^M$ and $\pi_h$.
\end{proof}

We end this section with some results on the properties of the extended basis. To that end, let $I^L\subset I$ 
be the indices for the basis $\mcB^L$, and note that $\mcB^E =  \{ E_h \varphi_i : i \in I^L\}$ is a basis in $V^E_{h,p,k}$.  Using the notation 
\begin{equation}
\varphi_i^E = E_h \varphi_i 
\end{equation}
for the extended basis functions, we then have the expansion 
\begin{equation}
v = \sum_{i=1}^{N^L} \hatv_i \varphi_i^E
\end{equation}
of $v \in V^E_{h,p,k}$. We next present a lemma that collects the basic properties of
the extended basis functions and then we show an equivalence between the degrees of freedom norm 
and the $L^2$ norm. The latter result is crucial in the proof of bounds on the condition number of stiffness 
and mass matrices.

\begin{lem}[Properties of Extended Basis Functions] There are constants such that 
\begin{equation}\label{eq:phiEprop}
\diam (\supp (\varphi_i^E ) ) \lesssim h, \qquad 
\| \varphi_i^E \|_{L^\infty(\Omega_h)} \lesssim 1
\end{equation}
and, with $\delta_{ij} = 1$ if $\supp(\varphi_i^E) \cap \supp(\varphi_j^E)\neq \emptyset$ and $0$ otherwise, 
\begin{equation}\label{eq:phiEintersec}
\max_{i \in I^L} \sum_{j\in I^L} \delta_{ij} \lesssim 1
\end{equation}
\end{lem}
\begin{proof}
We first note that if there is no element $T \in S_h(\mcT_h^S)$ in the support 
$\supp(\varphi_i)$ then $\varphi_i^E = E_h \varphi_i = \varphi_i$. These are basis functions in the interior that 
are not affected by the extension and obviously satisfy the desired properties.  If on the other hand there is 
$T \in \mcT_h^S$ such that $S_h(T) \subset \supp(\varphi_i)$ then $\varphi_i^E \neq \varphi_i$. Using the definition (\ref{def:Eh}) of the extension operator 
$E_h$ we have $\varphi_i^E = \pi_h B_h \varphi_i$. We start by noting that the support of $B_h \varphi_i$ is 
given by 
\begin{equation}
\supp(B_h \varphi_i) = \cup_{T \in \mcT_h^L, T \subset \supp(\varphi_i)} M_T
\end{equation}
where $M_T$ is the macro element defined in (\ref{eq:macro}). Using the fact that $\diam(\supp(\varphi_i)) \lesssim h$ and $\diam(M_T) \lesssim h$ we conclude that 
\begin{equation}\label{eq:diamBhphi)}
\diam( \supp(B_h \varphi_i) ) \lesssim h
\end{equation}
Recalling the 
definition (\ref{eq:omegahT}) and properties of the domain $\omega_{h,T}$ that influences $\pi_h v$ on $T$, we finally find that 
\begin{align}
\supp(\varphi_i^E) \subset \cup_{\omega_{h,T} \cap  \supp(B_h \varphi_i ) \neq \emptyset} \omega_{h,T}
\end{align}
which in particular means that 
\begin{equation}
\diam( \supp (\varphi_i^E ) ) \lesssim h
\end{equation}
since $\diam(\omega_{h,T}) \lesssim h$ and $ \diam( \supp(B_h \varphi_i) ) \lesssim h $, see (\ref{eq:omegahTdiam}) and (\ref{eq:diamBhphi)} ).  

Next using an inverse estimate to pass from the max norm to the $L^2$ norm followed by the $L^2$ stability 
of $\pi_h$ we have for any element $T \subset \supp (\varphi_i^E$, 
\begin{align}
\| \pi_h B_h \varphi_i \|^2_{L^\infty(T)} \lesssim h^{-d} \| \pi_h B_h \varphi_i \|^2_{T} 
\lesssim h^{-d} \| B_h \varphi_i \|^2_{\omega_{h,T}} \lesssim \| B_h \varphi_i \|^2_{L^\infty(\omega_{h,T})} 
\end{align}
Using stability of polynomial extension we may estimate the right hand side as follows 
\begin{align}
\| B_h \varphi_i \|^2_{L^\infty(\omega_{h,T})}  \lesssim \| \varphi_i \|^2_{L^\infty(\supp(\varphi_i) \cap \Omegahl)}
\lesssim \| \varphi_i \|^2_{L^\infty(\Omega_h)} \lesssim 1
\end{align} 
Finally, (\ref{eq:phiEintersec}) follows from the fact that $\supp(\varphi_i^E)$, $i \in I^L$, is contained in a ball $B_\delta(x_i)$ 
with $\delta \lesssim h$ and $x_i$ the midpoint of a unique element $T_i$ and therefore $\| x_i - x_j \|_{\IR^d} \gtrsim h$. Thus $\supp(\varphi_i^E) \cap \supp(\varphi_j^E) \neq \emptyset$ for a uniformly bounded number of indices $j\in I^L$.
\end{proof}
%\todo{check how this relates to assumption A1}
\begin{lem}[Equivalence with the Degrees of Freedom Norm] \label{lem:exth-Rn-eqv} Assume that the 
weights $\kappa_{T',i}$ in (\ref{def:pihcoeff}) satisfy (\ref{eq:weights-rest}).  Then there are constants 
such that 
\begin{align}
\boxed{ \| v \|^2_{\Omega} \sim h^{d} \| \hatv \|^2_{\IR^N } }
\end{align}
\end{lem}
\begin{proof} We recall that since the weights satisfy (\ref{def:pihcoeff}) we have  
\begin{equation}
(E_h v) |_{\Omegahl} = v|_{\Omegahl}
\end{equation}
see (\ref{eq:invariant-large}), which means that the extension does not change $v$ on $\Omegahl$. Then 
we can apply the equivalence on the element level (\ref{eq:BTstab}) on elements $T \in \mcT_h^L$, 
\begin{align} \label{eq:dofs-bound}
h^d \|\hatv\|^2_{\IR^N} \lesssim \sum_{T \in \mcT_h^L} h^d \|\hatv\|^2_{\IR^{N_T}}
\sim
\sum_{T \in \mcT_h^L} \|v\|^2_{T}  \lesssim_\gamma \sum_{T \in \mcT_h^L} \|v\|^2_{T\cap \Omega} 
=
\| v \|^2_\Omega
\end{align}
where we emphasized the dependency of $\gamma$ in the last inequality.
Conversely, using  (\ref{eq:phiEprop}) and (\ref{eq:phiEintersec}) it follows that   
\begin{align}
\| v \|^2_\Omega &= \sum_{i,j\in I^L} \hatv_i \hatv_j (\varphi_i^E,\varphi_j^E)_\Omega
\lesssim \sum_{i,j\in I^L} |\hatv_i|\, |\hatv_j|\, h^d \delta_{ij}
\\
&\qquad \lesssim \sum_{i\in I^L} |\hatv_i|^2 h^d \Big( \sum_{j\in I^L}\delta_{ij} \Big)
\lesssim \sum_{i\in I^L} |\hatv_i|^2 h^d
\end{align}
which completes the proof.
\end{proof}

\section{Application} \label{section:application}

To illustrate the application of the extension operator we consider an elliptic model problem 
with homogeneous boundary conditions in a domain $\Omega$ in $\IR^2$ with smooth 
boundary, find $u : \Omega \rightarrow \IR$ such that 
\begin{align}
-\Delta u = f \quad \text{in $\Omega$}, 
\qquad
u= 0 \quad \text{on  $\partial \Omega$}
\end{align}
We use Nitsche's method together with the extended spline space to discretize the problem. Since the below material is standard and only for illustration we only sketch the main arguments, for simplicity we assume $k=1$. The method 
takes the form: find $u_h \in V_{h,p,k}^E$ such that 
\begin{align}
a_h(u_h,v) = (f,v)_\Omega \,,\quad \forall v\in V_{h,p,k}^E
\end{align}
where 
\begin{align} \label{eq:nitsche-ah}
a_h(v,w) = (\nabla v, \nabla w )_\Omega - (\nabla_n v, w)_{\partial \Omega} - (v,\nabla_n w)_{\partial \Omega} 
+ \beta h^{-1} (v,w)_{\partial \Omega}
\end{align}
As is well known the key property required to show that $a_h$ is coercive on the trial space $V_{h,p,k}^E$ is 
the inverse inequality 
\begin{align}\label{eq:nitsche-inverse}
h \| \nabla_n v \|^2_{\partial \Omega} \lesssim \| \nabla v \|^2_\Omega,  \qquad v \in V_h^E
\end{align}
We note that for $T\in \mcT_h$ we have 
\begin{align}
h \| \nabla_n v \|^2_{\partial \Omega \cap T } \lesssim h |\partial \Omega \cap T|\,   \| \nabla v \|^2_{L^\infty(T)} 
\lesssim h |\partial \Omega \cap T| h^{-d} \| \nabla v \|^2_{T} 
\end{align}
Assuming,  $|\partial \Omega \cap T| h^{-(d-1)} \lesssim 1$ which holds for instance if the boundary is Lipshitz,  we 
get by summation over $T \in \mcT_h$, 
\begin{equation}
h \| \nabla_n v \|^2_{\partial \Omega \cap T } \lesssim \| \nabla v \|^2_{\Omega_h}
\end{equation} 
which combined with the stability of the extension operator in Lemma \ref{lem:exth-stab} directly establish 
the desired estimate (\ref{eq:nitsche-inverse}). It then follows using standard arguments that for $\beta$ large enough (depending on the constant in \eqref{eq:nitsche-inverse}),
\begin{align}
\tn v \tn^2 := \|\nabla v\|_\Omega^2 + \beta/h \|v\|_{\partial \Omega}^2 \lesssim a_h(v,v), \forall v \in V_{h,p,k}^E
\end{align}
If we let $e_h = u_h - \pi_h^E u$ we immediately see that $\tn u-u_h \tn \leq \tn u-\pi_h^E u \tn + \tn e_h \tn$ and 
\begin{align}
\tn e_h \tn^2 \lesssim a_h(e_h,e_h)
\end{align}
By the consistency of Nitsche's method $a_h(e_h,e_h) = a_h(u - \pi_h^E u, e_h)$. Using the Cauchy-Schwarz inequality and \eqref{eq:nitsche-inverse} we see that
\begin{align}
a_h(u - \pi_h^E u, e_h) \lesssim (\tn u - \pi_h^E u \tn +\|h^{\frac12} \nabla_n (u - \pi_h^E u)\|_{\partial \Omega})\tn e_h \tn
\end{align}
Using a trace inequality (see for instance \cite{WX19}) we see that
\begin{align}
\tn u - \pi_h^E u\tn  + \|h^{\frac12} \nabla_n (u - \pi_h^E u)\|_{\partial \Omega} \lesssim \sum_{i=0}^2 h^{i-1} \|D^i (u - \pi_h^E u)\|_{\Omega_h} 
\end{align}
where $D^i$ is the standard multi index notation for derivatives of order $i$.
Then we use the approximation property in Lemma~\ref{lem:exth-approx} 
\begin{align}
\|\nabla (u - \pi_h^E u)\|_\Omega + h^{-1} \|u - \pi_h^E u\|_{\Omega}+ h \|D^2(u - \pi_h^E u)\|_\Omega \lesssim h^{s-1} \| u \|_{H^{s}(\Omega)}
\end{align}
and obtain the optimal order a priori error estimate for $u \in H^r(\Omega)$,
\begin{align}
\tn u - u_h\tn   \lesssim h^{s-1} \| u \|_{H^{s}(\Omega)}, \quad s = \min(r,p+1)
\end{align}
We can also show, following the ideas of \cite{EG06}, using Lemma \ref{lem:exth-Rn-eqv} and the properties of the extension, that the condition number $\kappa$ of the stiffness matrix associated with the form $a_h$ 
satisfies the estimate 
\begin{align} \label{eq:condest}
\kappa \lesssim h^{-2}
\end{align}
\begin{rem}
Observe that it is straightforward to design and analyze a similar method in the case of fourth-order elliptic PDE, indeed the ideas of \cite{BHL2020} carry over verbatim if the ghost penalty terms are omitted.
\end{rem}

\section{Numerical Experiments} 

In this section, we present a number of numerical experiments in 2D, where a high-resolution polygonal domain is immersed in a structured quadrilateral mesh equipped with full regularity B-spline basis functions. For details on implementation, see \cite{MR3682761}.

\paragraph{Extension Choices.}
To construct the extension operator according to the above description, we must make the following three choices:
\begin{itemize}
\item \emph{The value of the parameter $\gamma\geq 0$.}
This parameter determines the partition of the mesh into large elements $\mcT_h^L$ respectively small elements $\mcT_h^S$ according to \eqref{eq:Largeelem}, which in turn gives the partition of the active basis functions into large basis functions $\mcB^L$ respectively small basis functions $\mcB^S$. The partition of the elements in one specific example is illustrated in Figure~\ref{fig:mesh-and-element-association-b} and the corresponding partition of the basis functions in Figure~\ref{fig:basis-decomposition}.

\item \emph{The small-to-large element mapping $S_h:\mcT^S \to \mcT^L$.}
This mapping defines the extension operator $B_h$ according to \eqref{def:Bh}.
In our experiments, we construct $S_h$ by associating each small element $T$ with the large element $T'$ that minimizes the distance between the center of mass of $T \cap \Omega$ and the center of mass of $T' \cap \Omega$. This is illustrated by the arrows in Figure~\ref{fig:mesh-and-element-association-b}.

\item \emph{The weights $\kappa_{T,i}$ in the definition of the interpolant.}
In the extension $E_h v = \pi_h B_h v$ these weights determine the coefficient for $\varphi_i$, defining how much weight $(B_h v)|_T$  should be given for all $T$ in the support of $\varphi_i$.
For our numerical examples we for each large basis function $\varphi_i \in \mcB^L$ set the weights for the elements its support to
\begin{align}
\kappa_{T,i} =
\left\{
\begin{alignedat}{2}
&|T \cap \Omega| , &\qquad& T \in \mcT_h^L
\\
&0 , &\qquad& T \in \mcT_h^S
\end{alignedat}\right.
\end{align}
while we for each small basis function $\varphi_i \in \mcB^S$ set the weights for the elements its support to $\kappa_{T,i} = |T \cap \Omega|$,
whereafter we normalize the weights such that $\sum_{T \in \mcT_{h,i}} \kappa_{T,i} = 1$. This choice fulfills \eqref{eq:weights-rest}, which means that the coefficients for large basis functions are unaffected by the extension.
\end{itemize}
As an illustration of how this realization of the extension couple small basis functions to large basis functions, we present the supports of some extended basis functions in Figure~\ref{fig:extended-basis-support}. Note that the above choices directly affect this outcome, where for instance an average with only a single non-zero entry per basis function would produce fever couplings between any small basis functions and large basis functions.

\begin{figure}\centering
\begin{subfigure}[t]{0.32\linewidth}\centering
\includegraphics[width=0.8\linewidth]{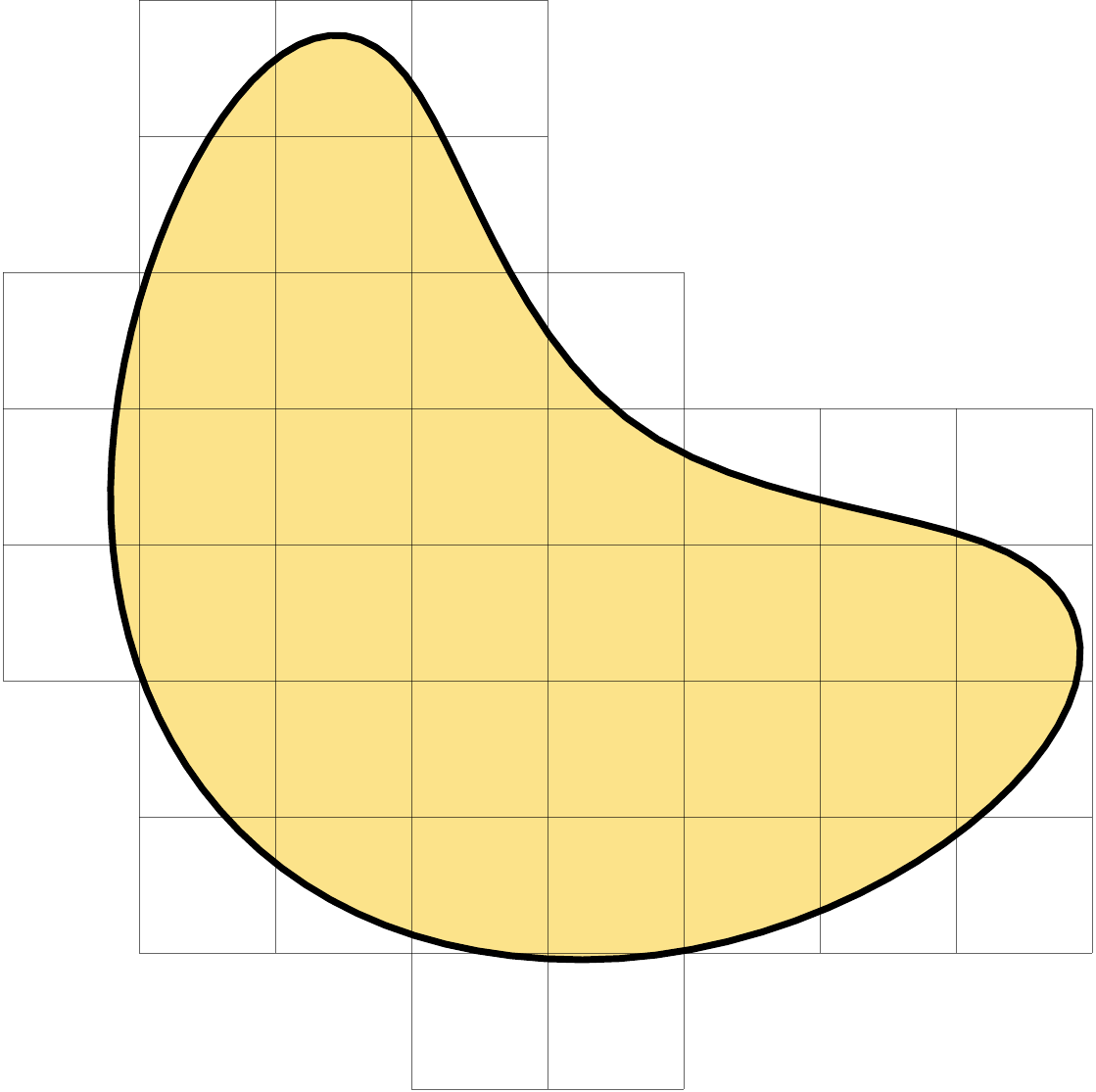}
\subcaption{Immersed domain}
\label{fig:mesh-and-element-association-a}
\end{subfigure}
\begin{subfigure}[t]{0.32\linewidth}\centering
\includegraphics[width=0.8\linewidth]{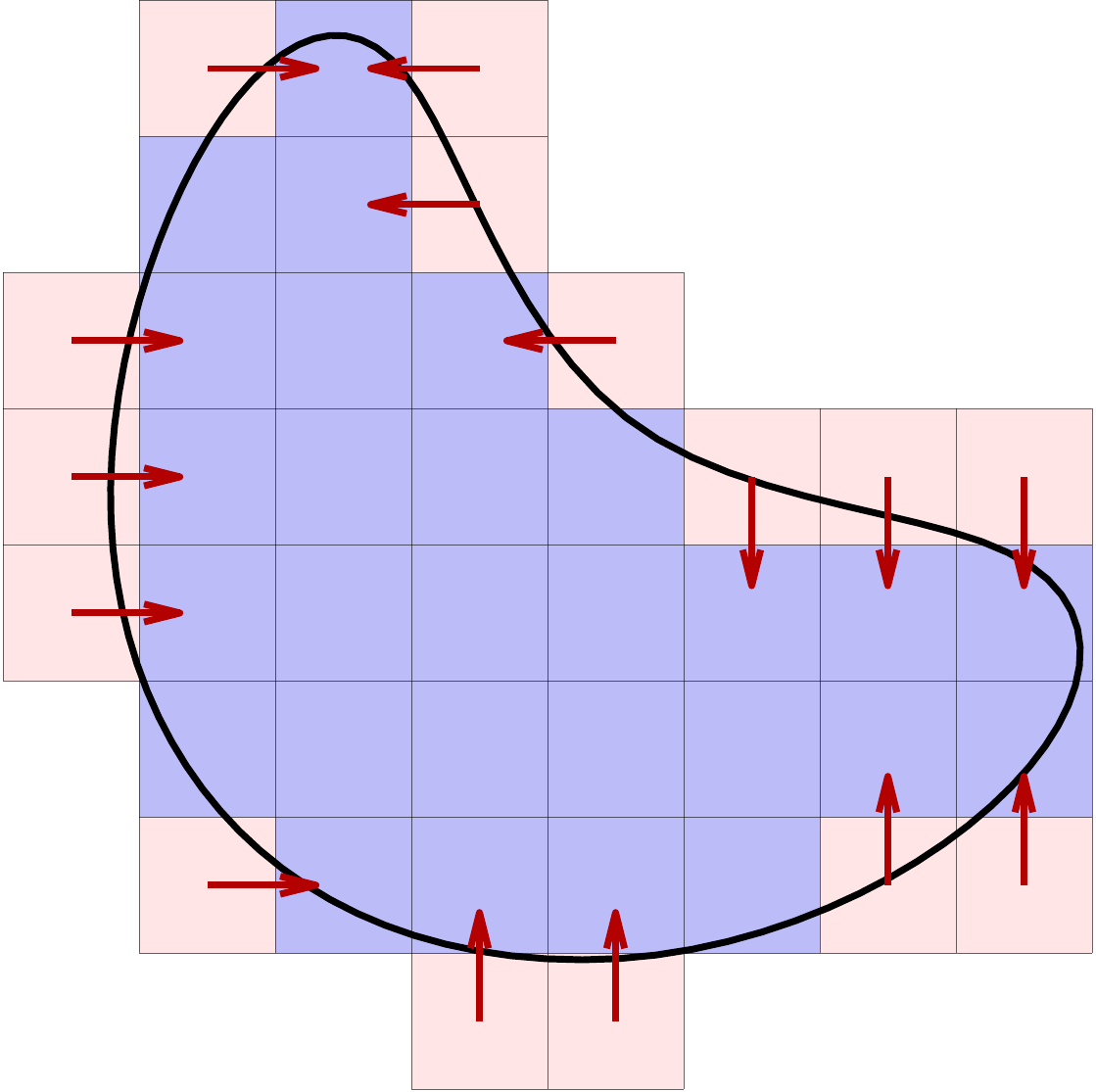}
\subcaption{Element association}
\label{fig:mesh-and-element-association-b}
\end{subfigure}

\caption{
\emph{Mesh and element association.}
(a) The domain $\Omega$ and the active mesh $\mcT_h$.
(b) The case $\gamma=0.5$ with corresponding large elements colored purple and small elements colored pink. The small-to-large element association $S_h$ is illustrated by red arrows.
}
\label{fig:mesh-and-element-association}
\end{figure}

\begin{figure}\centering
\begin{subfigure}[t]{0.32\linewidth}\centering
\includegraphics[trim=120px 40px 100px 20px, clip,width=0.935\linewidth]{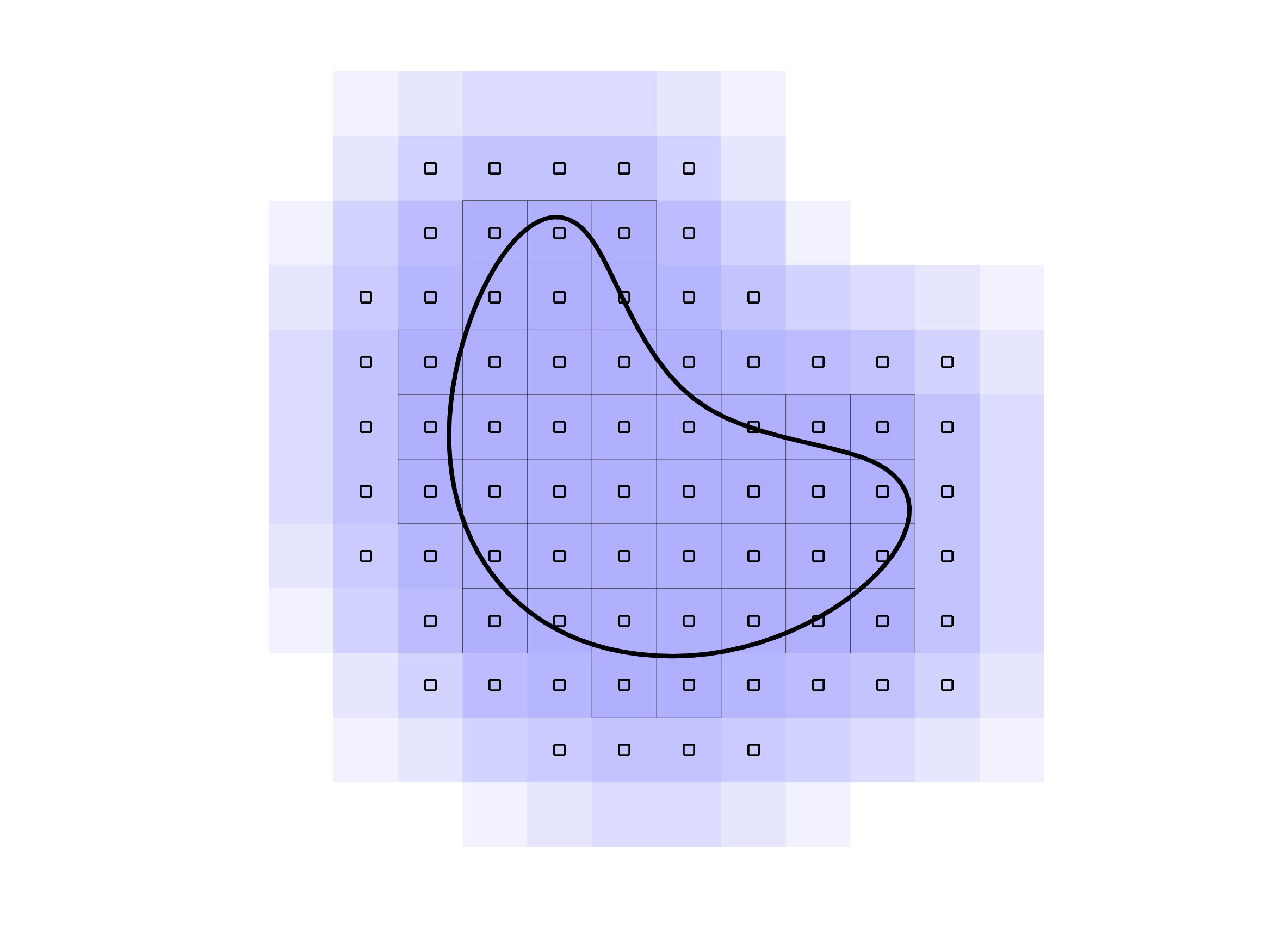}
\subcaption{All basis functions}
\end{subfigure}
\begin{subfigure}[t]{0.32\linewidth}\centering
\includegraphics[trim=120px 40px 100px 20px, clip,width=0.935\linewidth]{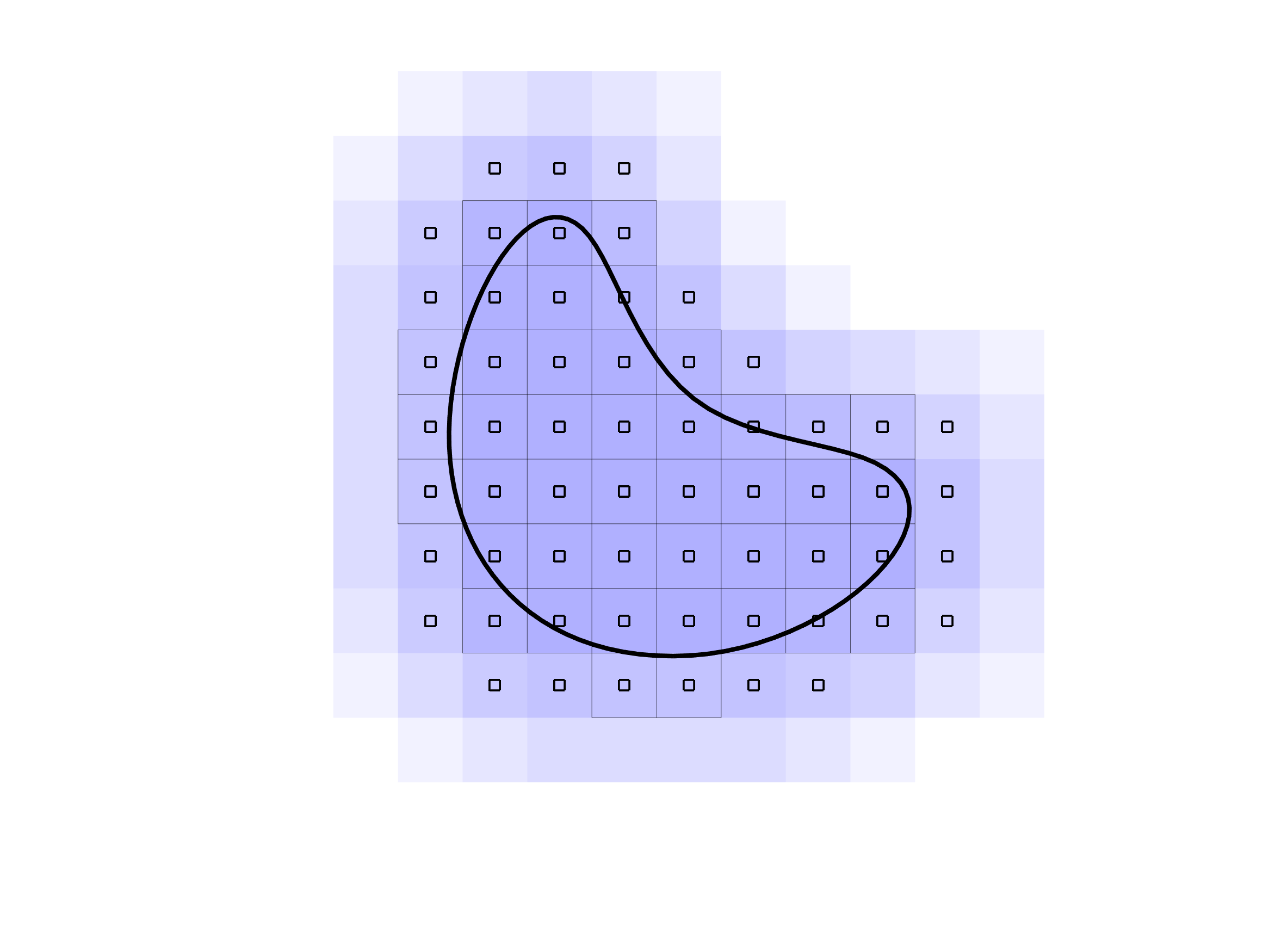}
\subcaption{Large basis functions}
\end{subfigure}
\begin{subfigure}[t]{0.32\linewidth}\centering
\includegraphics[trim=120px 40px 100px 20px, clip,width=0.935\linewidth]{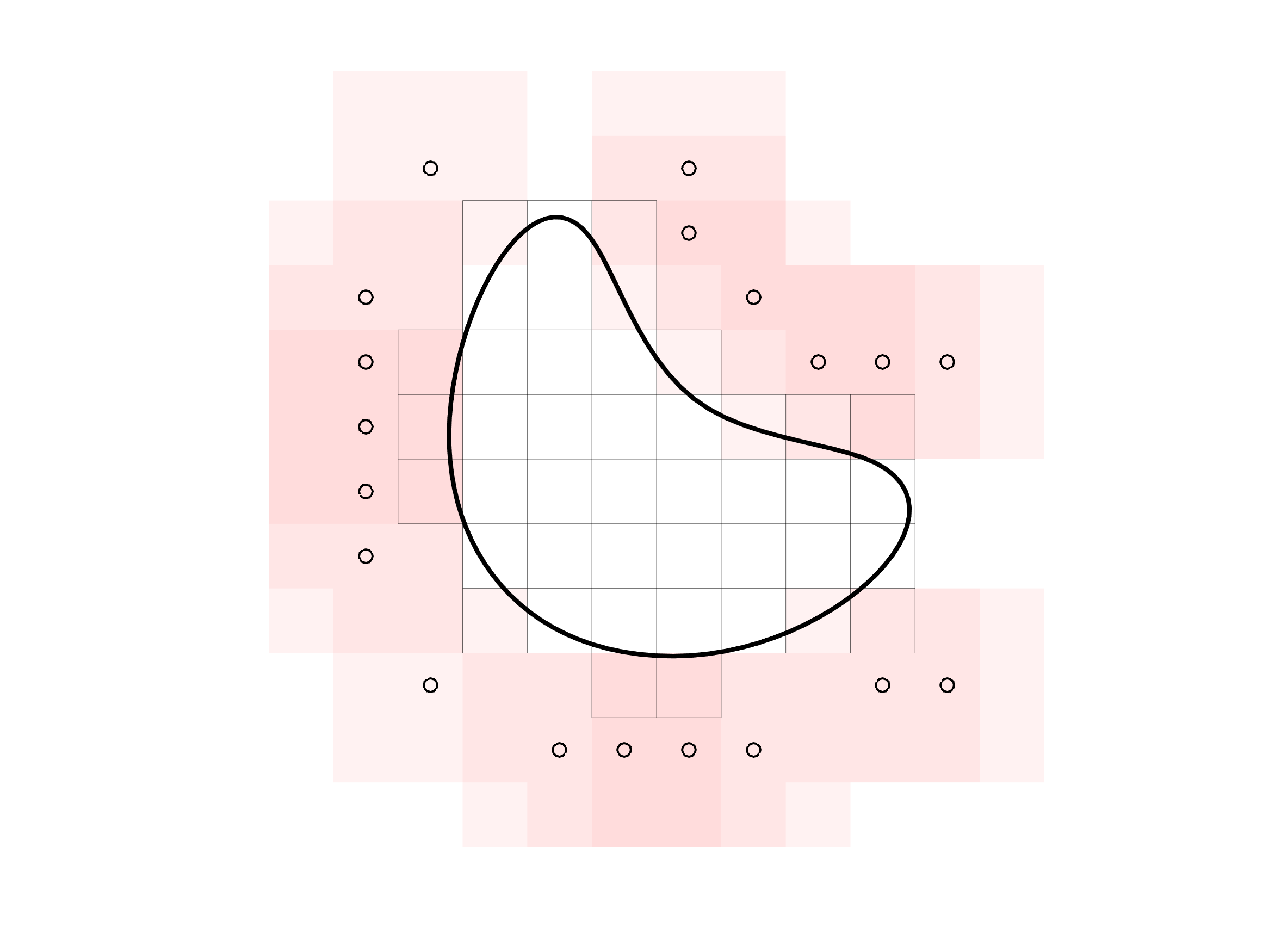}
\subcaption{Small basis functions}
\end{subfigure}

\caption{
\emph{Large and small basis functions.}
Partition of basis functions into large and small in the case $\gamma=0.5$. Here full regularity B-spline basis functions of order $p=2$ are used, and hence the support of a basis function will cover a $3\times 3$ block of elements in the mesh. In the plots, each basis function is indicated by shading of its support and a marker at the center of its support.
}
\label{fig:basis-decomposition}
\end{figure}

\begin{figure}\centering
\includegraphics[width=0.95\linewidth]{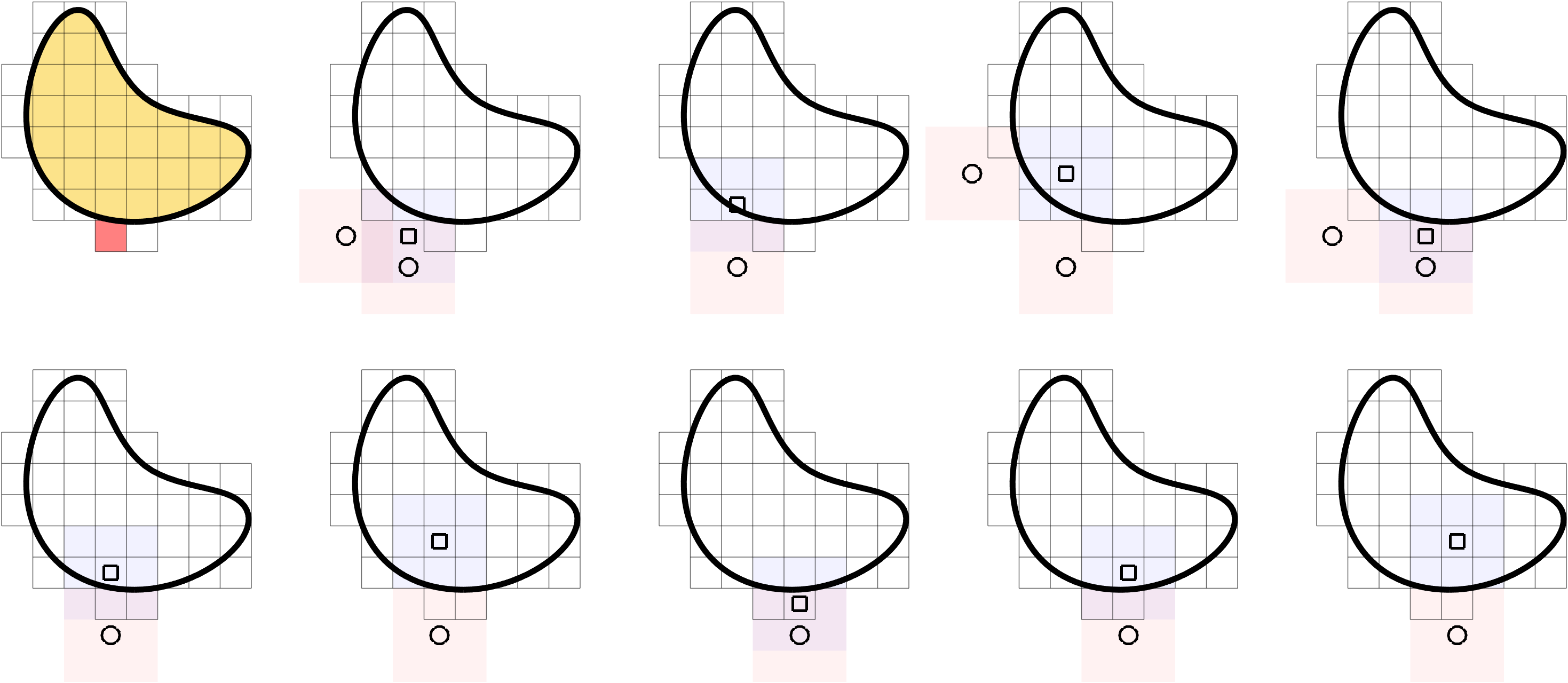}

\caption{
\emph{Extended basis support.}
Illustrations of extended basis functions ($\gamma=0.5$) which incorporate small basis functions with support in the element indicated in the upper left subplot. In each subplot the center of the original large basis function is indicated by a square and its support is shaded in light blue while the centers of associated small basis functions are indicated by circles and their supports are shaded in light pink.
}
\label{fig:extended-basis-support}
\end{figure}

\paragraph{Work Flow.}

In practice the extension is applied to our method on a linear algebra level, by the following steps:
\begin{itemize}
\item Assemble the square linear system of equations for the original method
\begin{align}
\hatA \hatu = \hatb
\end{align}
where $\hatu \in \IR^{\dim(V_{h,p,k})}$ are coefficients for the full approximation space $V_{h,p,k}$.

\item Given $\gamma\geq 0$, assemble the extension matrix $\hatE_h \in \IR^{\dim(V_{h,p,k})\times\dim(V_{h,p,k}^L)}$ that maps the large degrees of freedom onto all degrees of freedom. An example pseudocode description of this assembly is provided in Figure~\ref{alg:extensionMatrix}. In case $\gamma=0$ the extension matrix $\hatE_h$ reduces to the identity matrix.

\item Solve the reduced system
\begin{align}
(\hatE_h^T \hatA \hatE_h) \hatu_\gamma =  \hatE_h^T \hatb
\end{align}
where $\hatu_\gamma \in \IR^{\dim(V_{h,p,k}^L)}$ are coefficients for the extended space $V_{h,p,k}^E$.

\item Expand $\hatu_\gamma$ in coefficients for the full approximation space $V_{h,p,k}$ via
\begin{align}
\hatu^E =  \hatE_h \hatu_\gamma
\end{align}

\end{itemize}

\begin{figure}
{
\small
\begin{algorithmic}
\Procedure{ExtensionMatrix}{$\gamma$}

\State $\mcT_h^L \gets \emptyset$, $\mcB^L \gets \emptyset$
\Comment{Partition of elements and basis functions}
\ForAll{$T \in \mcT_h$}
\If{$\gamma h^d \lesssim |T\cap\Omega|$}
\State $\mcT_h^L \gets \mcT_h^L \cup T$,
 $\mcB^L \gets \mcB^L \cup \bigl( \cup_{i=I_T} \varphi_i \bigr)$
\EndIf
\EndFor
\State $\mcT_h^S \gets \mcT_h \setminus \mcT_h^L$, $\mcB^S \gets \mcB \setminus \mcB^L$
\vskip 7pt

%\Statex

\ForAll{$T \in \mcT_h^S$}
\Comment{Small-to-large element association}
%\If{$T \in \mcT_h^S$}
\State $S_h(T) \gets \arg \min_{T'\in\mcT_h^L} \mathrm{dist}\Bigl( |T\cap\Omega|^{-1} \int_{T\cap\Omega}x \,dx , |T'\cap\Omega|^{-1} \int_{T'\cap\Omega}x \,dx \Bigr)$
%\Else
%\State $S_h(T) \gets T$ 
%\EndIf
\EndFor
\vskip 7pt

%\Statex

\State $\widehat B_h \gets \mathrm{sparse}(|\mcB^{\mathrm{dG}}|,|\mcB|)$
\Comment{Preliminary extension from $V_{h,p,k}^L$ to a dG-space}
\ForAll{$T \in \mcT_h$}
\If{$T \in \mcT_h^S$}
\State $T' \gets S_h(T)$
\For{$i\in I_{T}^\mathrm{dG}$}
\State \textbf{set} $\widehat c\in\IR^{1\times |\mcB_{T'}|}$ such that $\varphi_i|_T = \sum_{j=1}^{|\mcB_{T'}|} \widehat{c}_j \varphi_{I_{T'}(j)}$
\State $\widehat B_h(i,I_{T'}) \gets \widehat c$ 
\EndFor
\Else
\State $\widehat B_h(I_T^\mathrm{dG},I_T) \gets \mathrm{Id}_{|\mcB_T|}$ 
\EndIf
\EndFor
\State $(\widehat B_h)_{*,I^S} \gets \emptyset$ %\Comment{Remove columns corresponding to $\mcB^S$}
\vskip 7pt

%\Statex

\State $\widehat I_h \gets \mathrm{sparse}(|\mcB|,|\mcB^{\mathrm{dG}}|)$
\Comment{Interpolation from the dG-space to $V_{h,k,p}$}
\ForAll{$i\in I$}
\ForAll{$T \in \mcT_{h,i}$}
\If{$\varphi_i \in \mcB^L$ \textbf{and} $T \in \mcT_h^S$}
\State $\kappa_{T,i} \gets 0$
\Else
\State $\kappa_{T,i} \gets |T \cap \Omega|$
\EndIf
\EndFor
\State \textbf{normalize} such that $\sum_{T \in \mcT_{h,i}} \kappa_{T,i} = 1$
\ForAll{$T \in \mcT_{h,i}$}
\State \textbf{find} $j$ such that $I_T(j)=i$
\State $(\widehat I_h)_{i,I^\mathrm{dG}_T(j)} \gets \kappa_{T,i}$
\EndFor
\EndFor
\vskip 7pt

%\Statex

\State \textbf{return} $\widehat E_h \gets \widehat I_h \widehat B_h$
\Comment{Final extension matrix}
\EndProcedure
\end{algorithmic}
}
\caption{\emph{Extension matrix assembly.} For a convenient description we here assume that $B_h$ extends into a discontinuous Galerkin version of the space $V_{h,p,k}$, discontinuous between all elements. We let previously used notations extend to this dG-space, signified by superscript `dG'. The local ordering of the basis functions in the elementwise index sets are assumed fixed.}\label{alg:extensionMatrix}
\end{figure}

\paragraph{Model Problem.} We consider the non-homogeneous version of the application examined in Section~\ref{section:application}, i.e., we solve the Dirichlet problem
\begin{align}
-\Delta u = f \quad \text{in $\Omega$}, 
\qquad
u= g_D \quad \text{on  $\partial \Omega$}
\end{align}
using Nitsche's method:
find $u_h \in V_{h,p,k}^E$ such that 
\begin{align} \label{eq:example-method}
a_h(u_h,v) = l_h(v) \,,\quad \forall v\in V_{h,p,k}^E
\end{align}
with $a_h$ given by \eqref{eq:nitsche-ah} and
\begin{align}
l_h(w) = - (g_D,\nabla_n w)_{\partial \Omega} 
+ \beta h^{-1} (g_D,w)_{\partial \Omega}
\end{align}
In all experiments, we use the Nitsche penalty parameter $\beta=25 p^2$. Note that this method does not include any additional stabilization terms for ensuring optimal stability properties regardless of how the domain cuts through the mesh, cf. \cite{Bu10,BCHLM15,MR4394710}.

In our quantitative experiments, we use the bean shaped geometry in Figure~\ref{fig:mesh-and-element-association-a} as our domain $\Omega$, whose boundary $\partial\Omega$ is constructed as the cubic spline interpolation of the periodic angular data points
\begin{align}
\theta &=
\left\{0, -\pi/2, \pi/20, \pi/4, \pi/2, \pi, 3\pi/2,0 \right\}
\\
x &= 
\left\{
\begin{psmallmatrix} 1 \\ 0 \end{psmallmatrix},
\begin{psmallmatrix} 0 \\ -0.8 \end{psmallmatrix},
\begin{psmallmatrix} 0.7 \\ -0.1 \end{psmallmatrix},
\begin{psmallmatrix} 0.1 \\ 0.1 \end{psmallmatrix},
\begin{psmallmatrix} -0.3 \\ 0.7 \end{psmallmatrix},
\begin{psmallmatrix} -0.8 \\ 0 \end{psmallmatrix},
\begin{psmallmatrix} 0 \\ -0.8 \end{psmallmatrix},
\begin{psmallmatrix} 1 \\ 0 \end{psmallmatrix}
\right\}
\end{align}
A problem with known analytical solution is manufactured by deriving
$f$ and $g_D$ from the ansatz $u=\frac{1}{10}\left(\sin(2x) + x\cos(3y) \right)$, and we present a numerical solution to this problem in Figure~\ref{fig:solution}.

\paragraph{Worst Case Mesh.}
We embed the domain $\Omega$ in a structured quadrilateral mesh equipped with B-spline basis functions of order $p$ and maximum regularity $k=p-1$. To mitigate the effect of chance in the cut situations generated in our convergence and stability studies we for each mesh size $h$ generate a sequence of 100 meshes by shifting the background mesh and then report the worst result from this sequence.

\begin{figure}\centering
\includegraphics[width=0.35\linewidth]{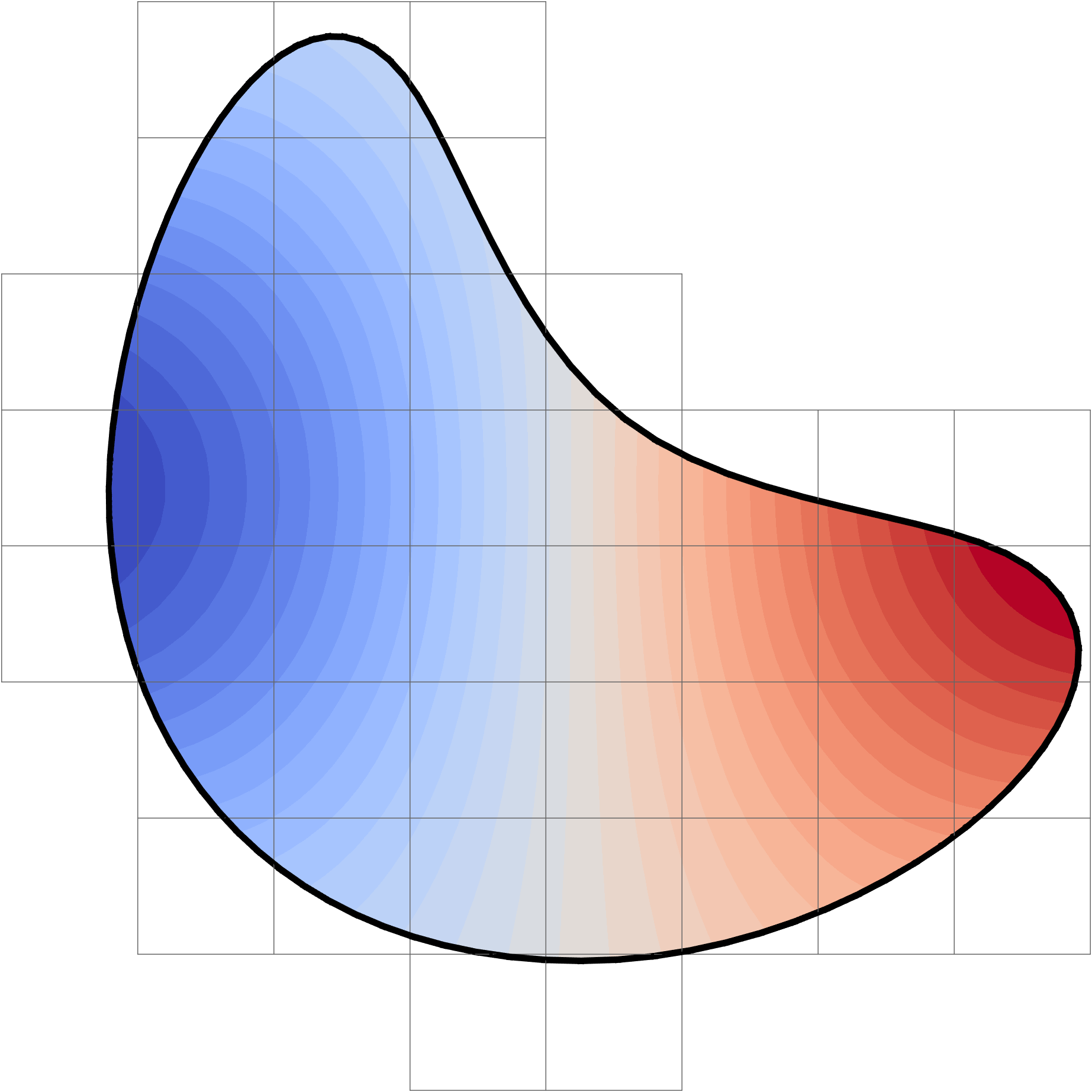}
\caption{
\emph{Numerical solution.}
Model problem solution using extension with $\gamma=0.5$ and full regularity B-spline basis function of order $p=2$.
}
\label{fig:solution}
\end{figure}

\paragraph{Convergence.} 
In Figure~\ref{fig:convergence} we collect our convergence studies:
\begin{itemize}
\item
Firstly, in (a)--(b) we consider convergence using $p=2$ basis functions and extension for a wide range of $\gamma\in[0,1]$. We note poor convergence results in the case without extension ($\gamma=0$), likely due to loss of coercivity in the method in the worst cut situation. We also note that for larger $\gamma$ the errors are initially somewhat higher, which seems reasonable since then more basis functions close to the boundary are extended.

\item
Secondly, in (c)--(d) we consider convergence using $p=1,2,3$ basis functions and extension with $\gamma=1$ and note that optimal order convergence seems to be asymptotically obtained.
\end{itemize}

\begin{figure}\centering
\begin{subfigure}[t]{0.35\linewidth}\centering
\includegraphics[width=0.9\linewidth]{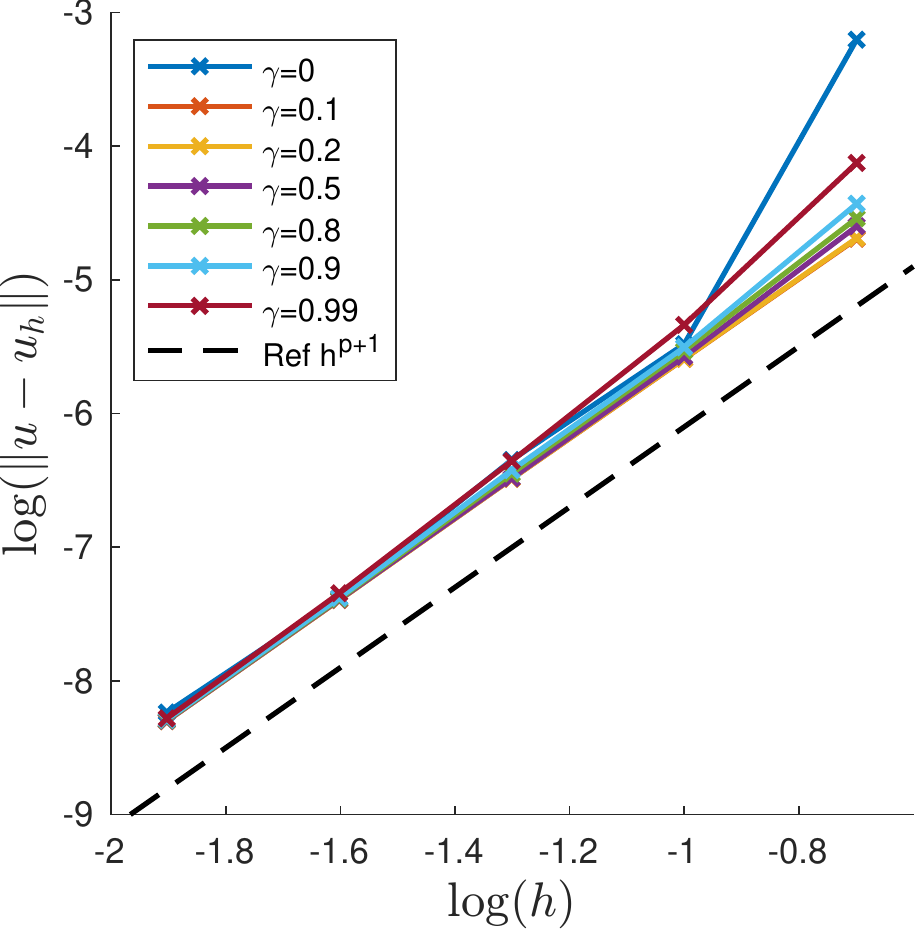}
\subcaption{$p=2$}
\end{subfigure}
\begin{subfigure}[t]{0.35\linewidth}\centering
\includegraphics[width=0.9\linewidth]{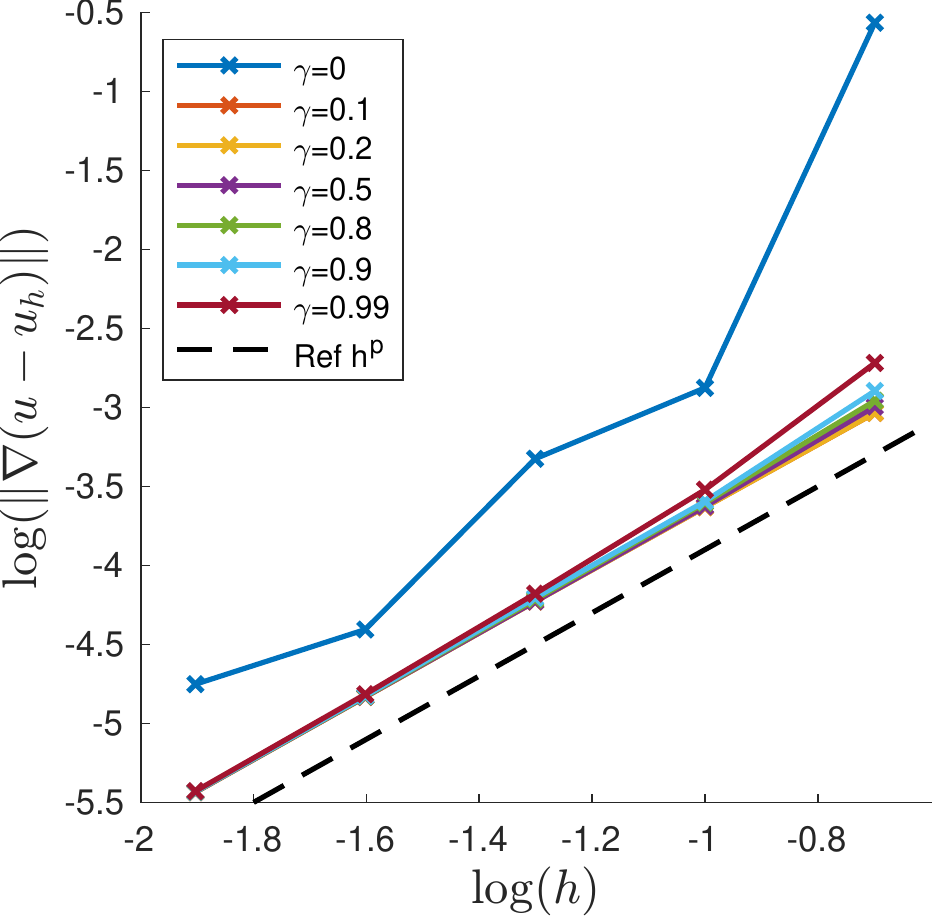}
\subcaption{$p=2$}
\end{subfigure}

\begin{subfigure}[t]{0.35\linewidth}\centering
\includegraphics[width=0.9\linewidth]{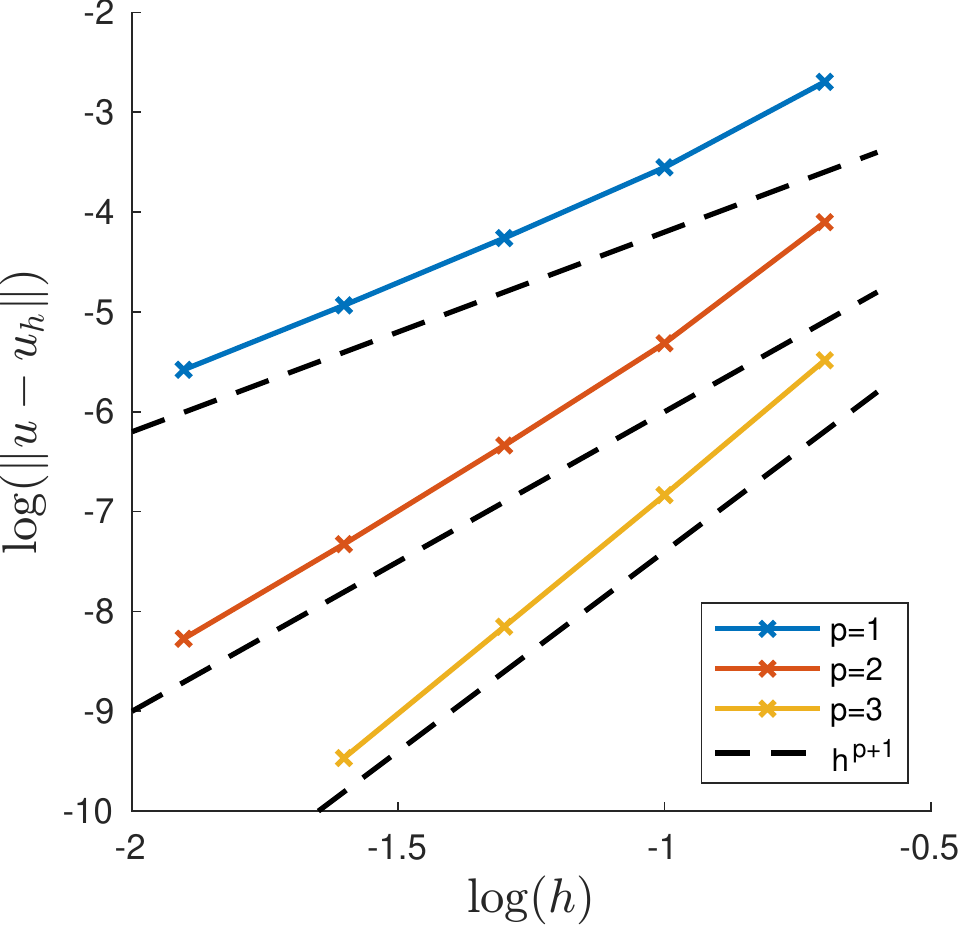}
\subcaption{$\gamma=1$}
\end{subfigure}
\begin{subfigure}[t]{0.35\linewidth}\centering
\includegraphics[width=0.9\linewidth]{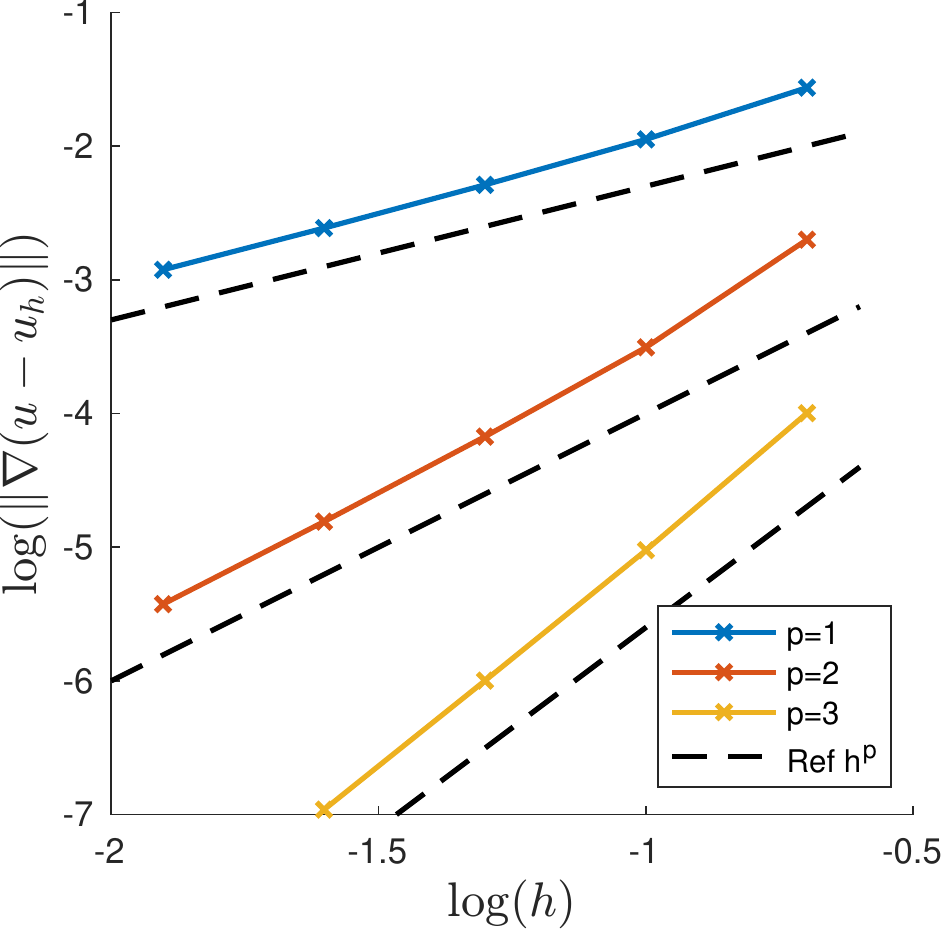}
\subcaption{$\gamma=1$}
\end{subfigure}

\caption{
\emph{Convergence.}
In these experiments, the largest error among 100 different cut situations is presented for each mesh size $h$.
\textbf{(a)--(b)}~Errors in $L^2$-norm respectively in $H^1$-seminorm when using extension for a wide range of $\gamma\in[0,1]$ and full regularity B-spline basis functions of order $p=2$.
\textbf{(c)--(d)}~Errors in $L^2$-norm respectively in $H^1$-seminorm using full regularity B-spline basis functions of orders $p=1,2,3$ and extension with $\gamma=1$.
}
\label{fig:convergence}
\end{figure}

\paragraph{Condition Number Scaling.}
To ensure stable computation of the condition number we convert our sparse matrix $A$ to a full storage matrix and employ dense linear algebra for the computation, i.e., using \verb+cond(full(A))+ in Matlab.
In Figure~\ref{fig:conditionnumber-scaling} we collect studies of how the condition number scales with $h$:
\begin{itemize}
\item
Firstly, in (a)--(b) we consider the condition number $h$-scaling using $p=2$ basis functions and extension for a wide range of $\gamma\in[0,1]$.
We note that without extension ($\gamma=0$) the condition number seemingly can become unbounded. Also, for $\gamma > 0$ the size of $\gamma$ seems to have the effect that for larger values, a smaller mesh size is required before entering the desired asymptotic $h^{-2}$ scaling found in \eqref{eq:condest}. This delay we attribute to the dependence of $\gamma$ in the bound \eqref{eq:dofs-bound}.

\item
Secondly, in (c) we consider condition number $h$-scaling using $p=1,2,3$ basis functions and extension with $\gamma=1$ and note that the desired $h^{-2}$ scaling is plausibly asymptotically obtained in all cases, even though that stage in not quite reached in the case of $p=3$.
\end{itemize}
In (d)--(f) we consider the same studies as in (a)--(c) albeit including simple preconditioning of the stiffness matrix using diagonal scaling.

\begin{figure}\centering
\begin{subfigure}[t]{0.32\linewidth}\centering
\includegraphics[width=0.99\linewidth]{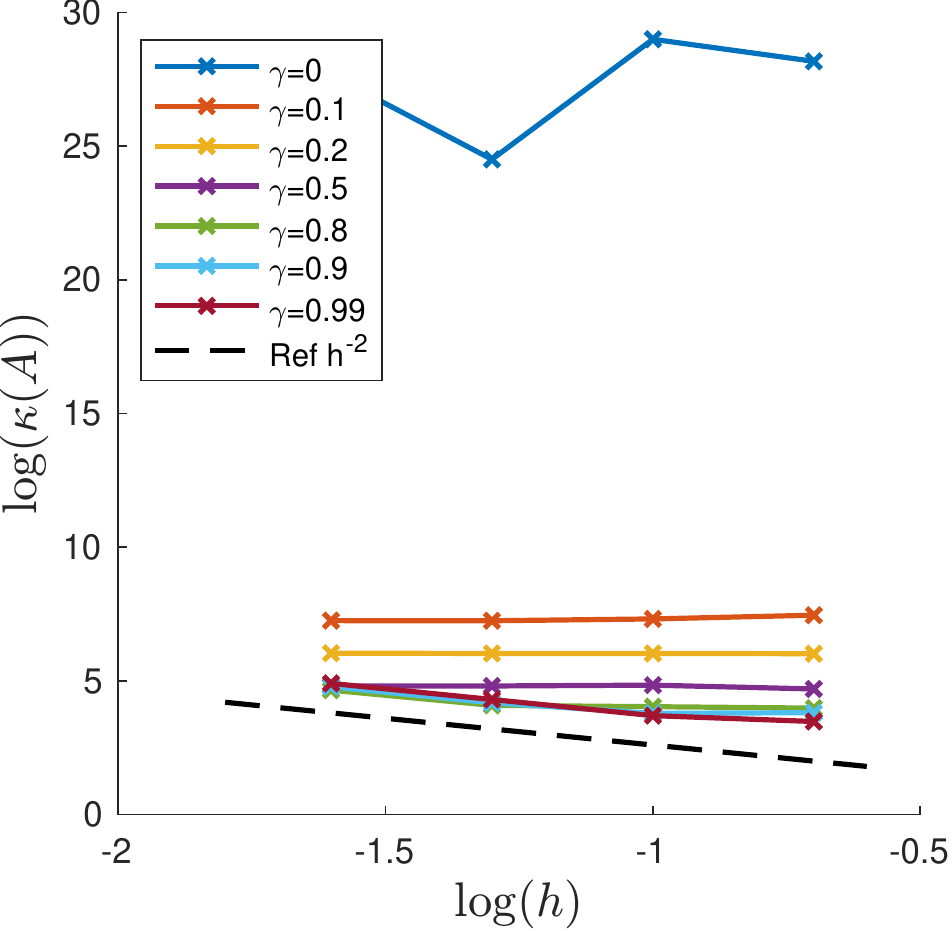}
\subcaption{$p=2$}
\end{subfigure}
\begin{subfigure}[t]{0.32\linewidth}\centering
\includegraphics[width=0.99\linewidth]{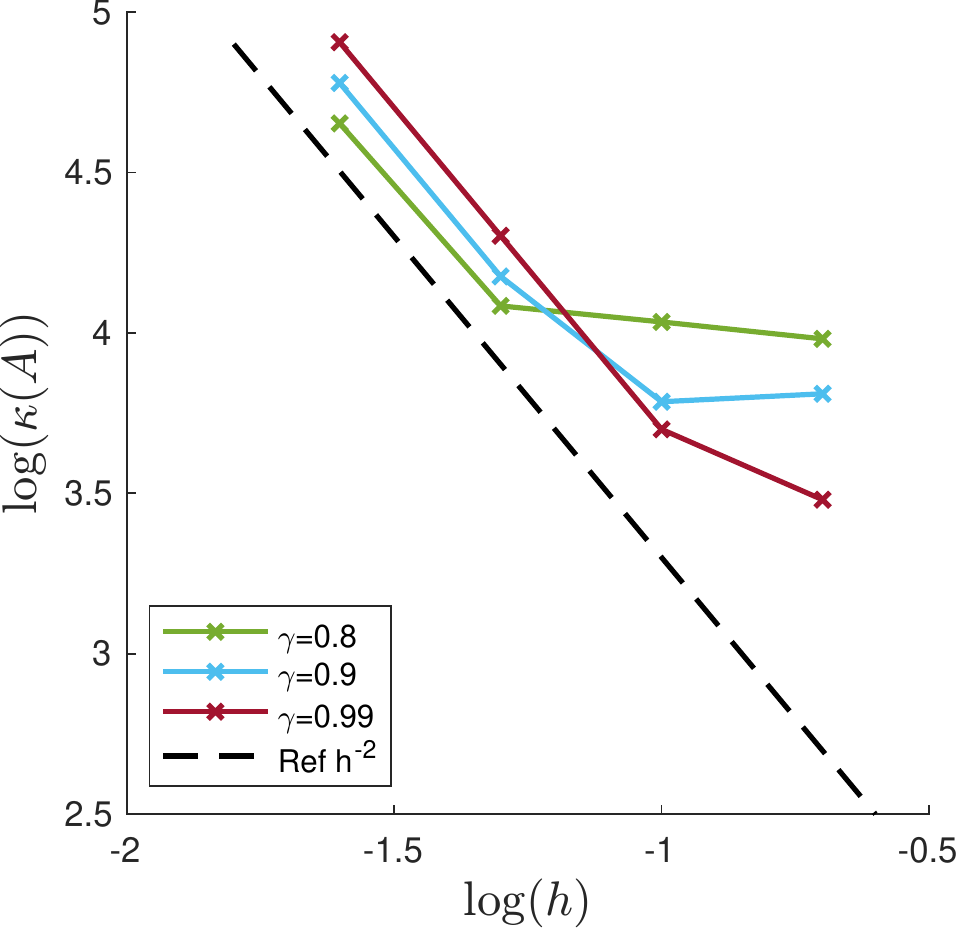}
\subcaption{$p=2$}
\end{subfigure}
\begin{subfigure}[t]{0.32\linewidth}\centering
\includegraphics[width=0.99\linewidth]{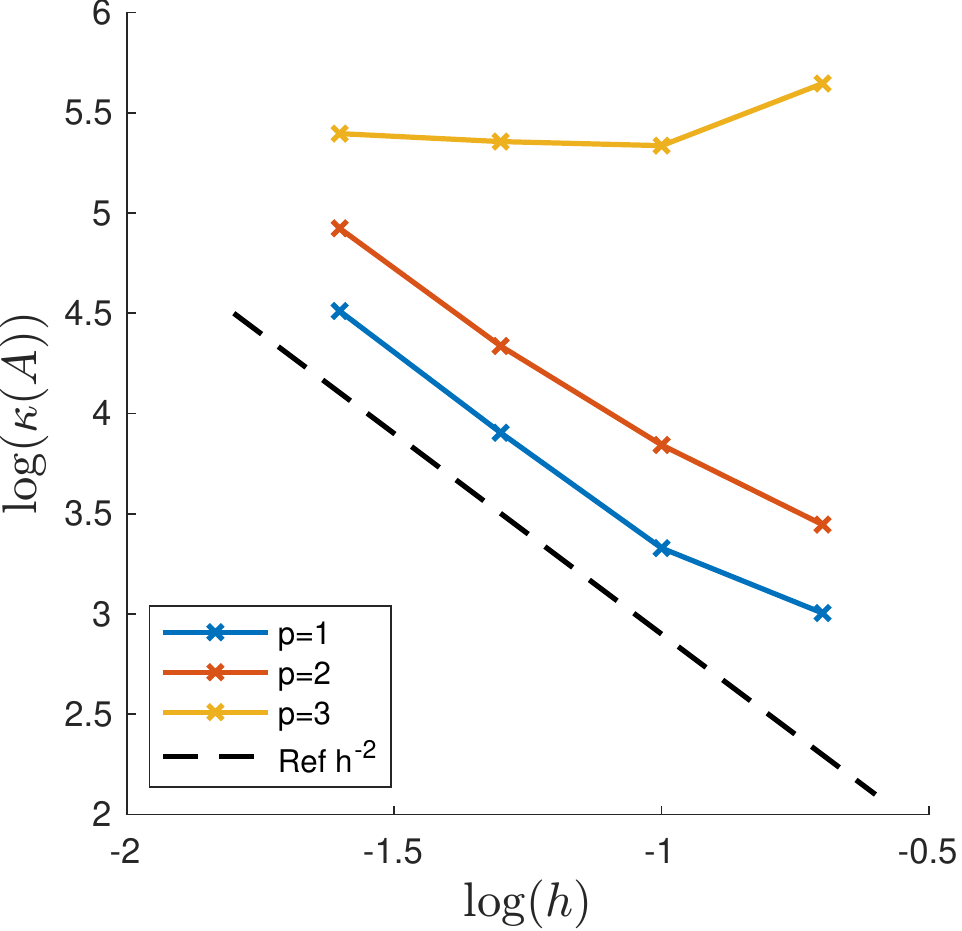}
\subcaption{$\gamma=1$}
\end{subfigure}

\begin{subfigure}[t]{0.32\linewidth}\centering
\includegraphics[width=0.99\linewidth]{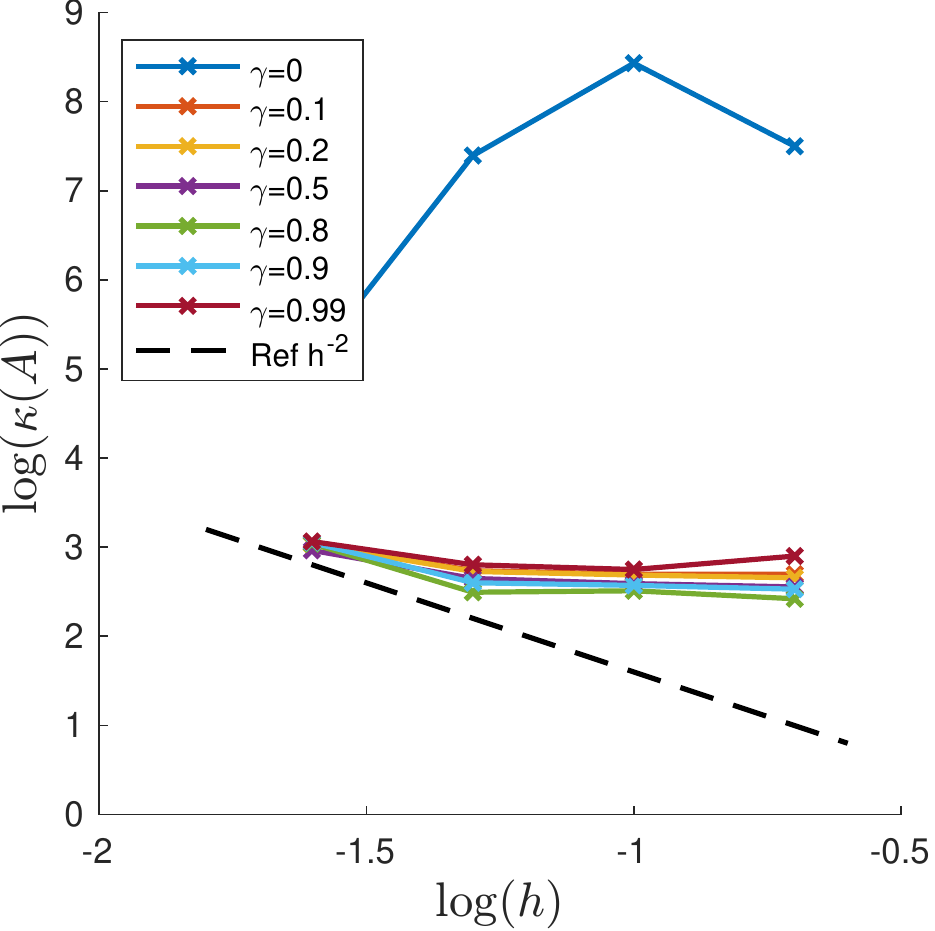}
\subcaption{$p=2$}
\end{subfigure}
\begin{subfigure}[t]{0.32\linewidth}\centering
\includegraphics[width=0.99\linewidth]{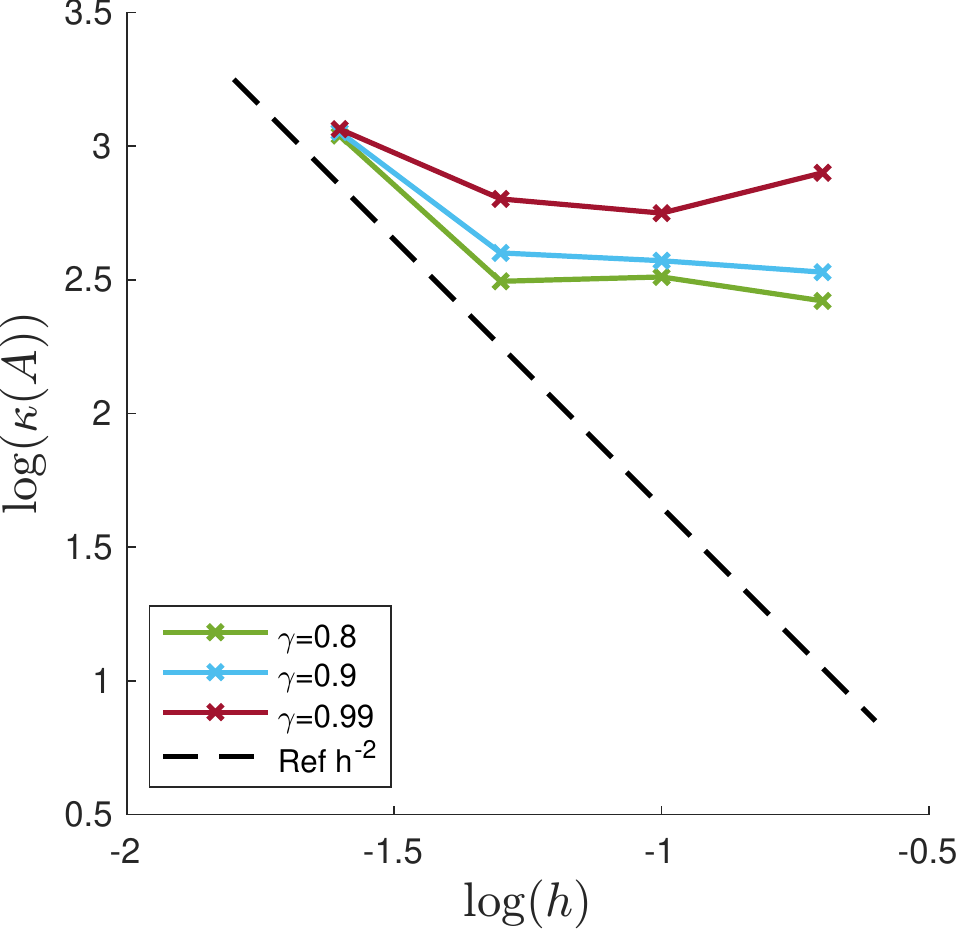}
\subcaption{$p=2$}
\end{subfigure}
\begin{subfigure}[t]{0.32\linewidth}\centering
\includegraphics[width=0.99\linewidth]{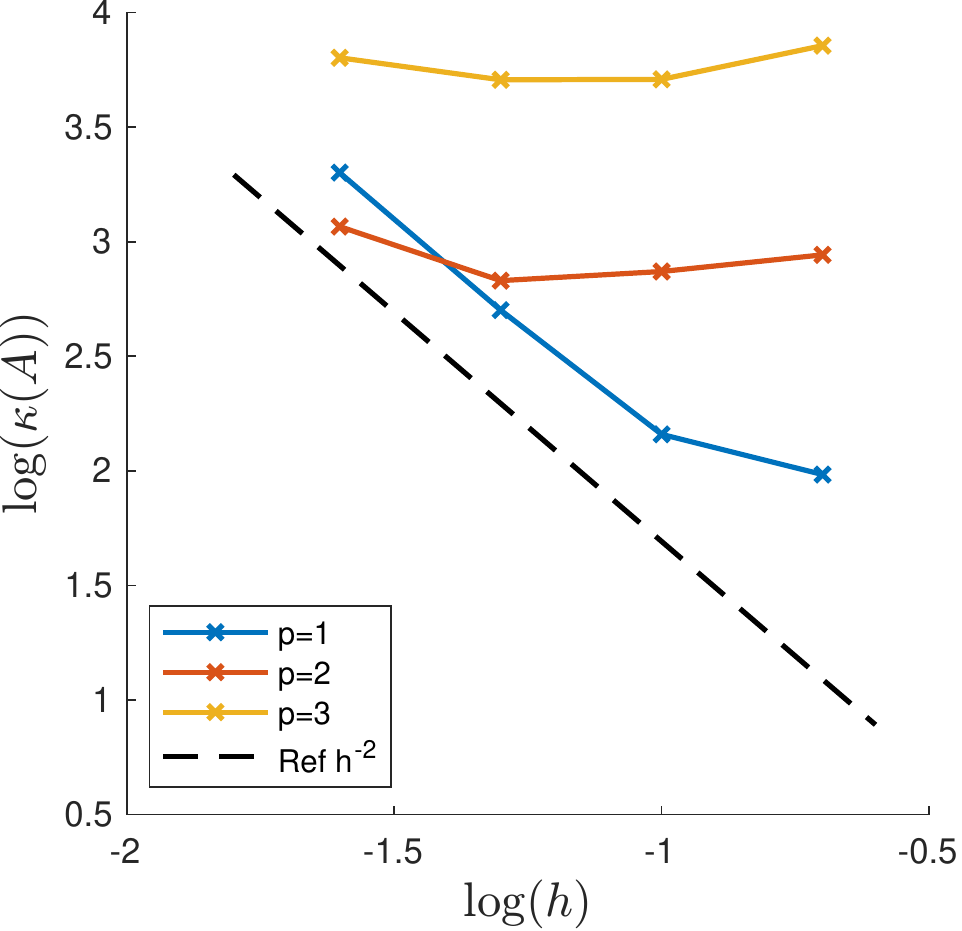}
\subcaption{$\gamma=1$}
\end{subfigure}

\caption{
\emph{Condition number $h$-scaling.} In these experiments the worst stiffness matrix condition number among 100 different cut situations is presented for each mesh size $h$.
\textbf{(a)}~Condition numbers when using extension for a wide range of $\gamma\in[0,1]$ and full regularity quadratic B-spline basis functions.
\textbf{(b)}~Detailed view of the study in (a) limited to the three largest values of $\gamma$.
\textbf{(c)}~Condition numbers using full regularity B-spline basis functions of orders $p=1,2,3$ and extension with $\gamma=1$.
\textbf{(d)--(f)}~Studies corresponding to (a)--(c) when the stiffness matrices are also preconditioned using diagonal scaling.
}
\label{fig:conditionnumber-scaling}
\end{figure}

\paragraph{Example on a Parametric Surface.}
Let $\Gamma\subset\IR^3$ be a surface described via the parametric mapping $\Phi: \Omega \to \Gamma \subset \IR^3$ with induced Riemann metric $G\in\IR^{2\times 2}$, $G_{ij}=\partial_i \Phi \cdot \partial_j \Phi$.
Using the lifting $v^\ell = v \circ \Phi^{-1}$
we can formulate the Dirichlet problem on $\Gamma$ as seeking an unknown on the reference domain $\Omega$; find $u:\Omega \to \IR$ such that
\begin{align}
-\Delta_\Gamma u^\ell &= f \quad  \text{in $\Gamma$}
\,, \qquad
u^\ell = g   \quad  \text{on $\partial\Gamma$}
\end{align}
where $\Delta_\Gamma$ is the Laplace-Beltrami operator on $\Gamma$.
Transforming this problem back to Euclidean coordinates 
in $\Omega$ we can derive a Nitsche's method for the problem defined by the forms
\begin{align} \label{eq:ref-forms-a}
a_h(v,w) 
&= \left( |G|^{1/2} \nabla v, G^{-1} \nabla w \right)_{\Omega}
\\ \nonumber
&\quad - \left( |G|^{1/2} n \cdot G^{-1} \nabla v, w \right)_{\partial \Omega}
+ \left( |G|^{1/2} v, \beta h^{-1} - n \cdot G^{-1} \nabla w \right)_{\partial \Omega}
\\ \label{eq:ref-forms-b}
l_{h}(v) &= \left( |G|^{1/2} f\circ\Phi, v \right)_{\Omega} 
 + \left( |G|^{1/2} g\circ\Phi,\beta {h}^{-1} v - n \cdot G^{-1}\nabla 
v \right)_{\partial \Omega}
\end{align}
Since the mesh $\mcT_h$ and approximation space are naturally defined on $\Omega$, rather than on the curved surface $\Gamma$, we can directly apply the extension to the resulting system of equations. An example solution is presented in Figure~\ref{fig:surf-example}, where a circular reference domain is mapped onto a cone.

\begin{figure}\centering
\begin{subfigure}[t]{0.32\linewidth}\centering
\includegraphics[width=0.8\linewidth]{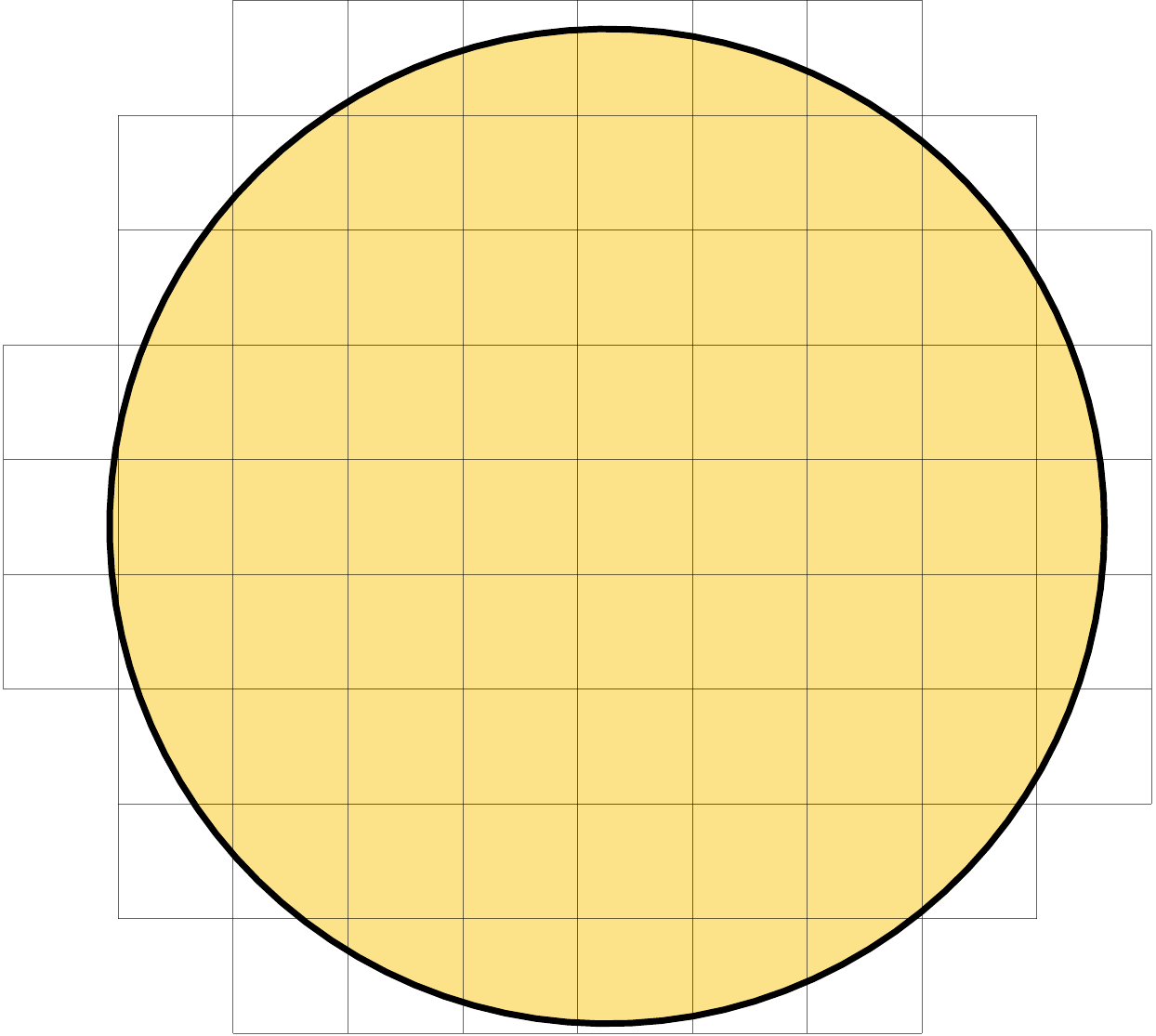}
\subcaption{Reference domain}
\end{subfigure}
\begin{subfigure}[t]{0.36\linewidth}\centering
\includegraphics[width=0.95\linewidth]{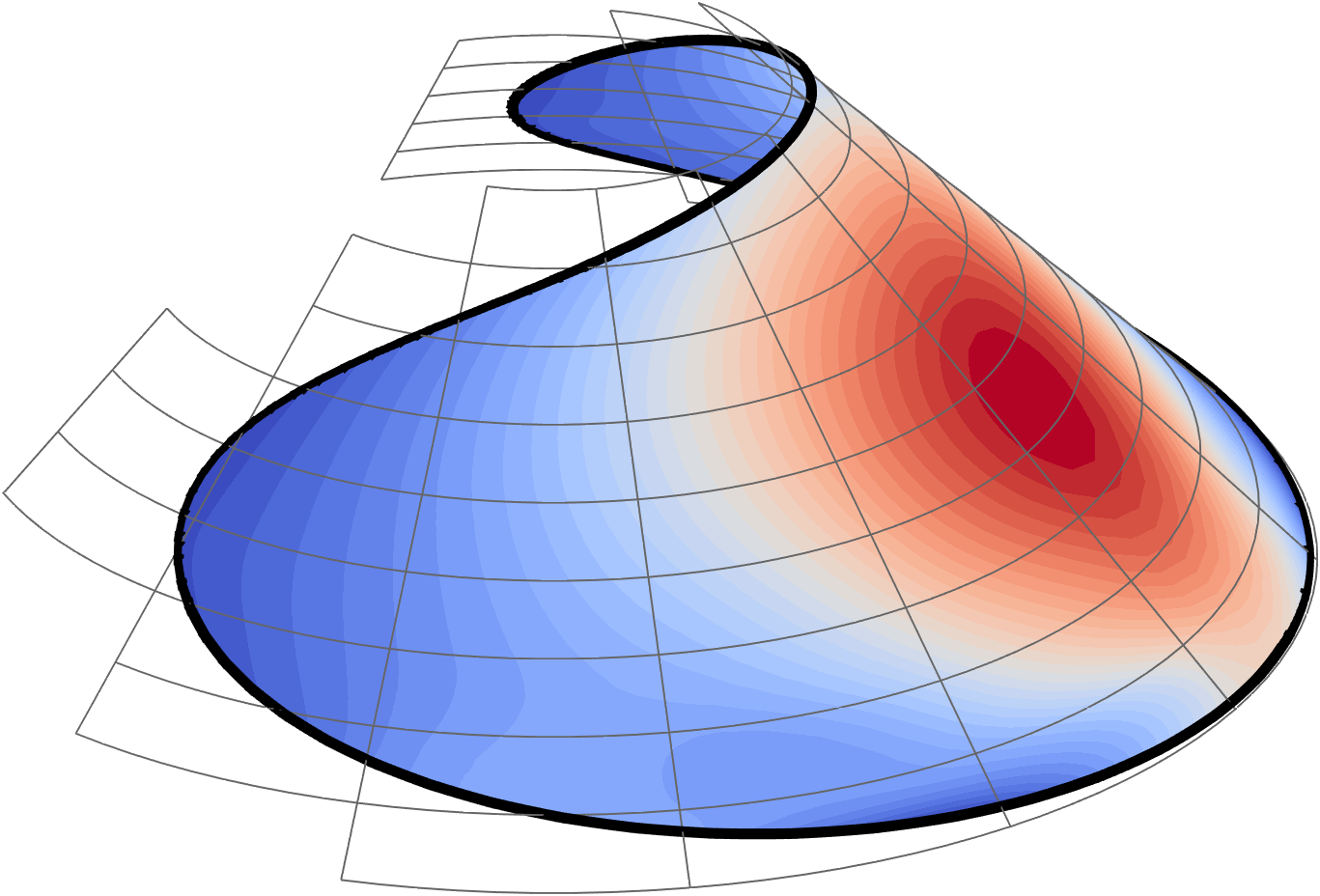}
\subcaption{Solution}
\end{subfigure}

\caption{
\emph{Extension on mapped domain.}
In this example, the domain is a curved surface $\Gamma$ constructed as a map from the unit square onto the surface of a cone, and a circular trim curve in the reference domain $\Omega$.
\textbf{(a)}~Reference domain $\Omega$ and the mesh on which the approximation space is defined.
\textbf{(b)}~Numerical solution for a Dirichlet problem using extension ($\gamma=0.5$).
}
\label{fig:surf-example}
\end{figure}

\section{Conclusions}

We provide a recipe for the construction of a family of extension operators for trimmed spline spaces that feature:
\begin{itemize}
\item \emph{Proven approximation and stability properties.} The extension gives additional stabilization by eliminating the ill-conditioned `small' degrees of freedom, expressing them in terms of well-conditioned `large' degrees of freedom while maintaining optimal order accuracy for the extended space. This is reflected in a good performance in our numerical experiments.

\item \emph{Convenient implementation.} The extension describes a mapping from large degrees of freedom onto all degrees of freedom and is conveniently applied to the linear algebra formulation of the method. Also, since the extension induces the necessary stability, special stabilization terms or other forms of manipulation are not required to ensure robustness with respect to how the spline space is trimmed.

\item \emph{Natural choices in the construction.} The choices needed for the construction of the extension operator are quite natural; the parameter $\gamma \geq 0$  determining the partition into large and small elements that in turn gives the partition into large and small basis functions, the mapping $S_h$ onto large elements, and the weights $w_{T,i}$ in the interpolation operator. During our numerical testing, we found no particular sensitivity in how these choices seem to affect performance.
\end{itemize}

\bibliographystyle{habbrv}
\footnotesize{
\bibliography{ref}
}

\paragraph{Acknowledgements.}
This research was supported in part by the Swedish Research
Council Grants Nos.\  2017-03911, 2018-05262,  2021-04925,  and the Swedish
Research Programme Essence. EB was supported in part by the EPSRC grant EP/P01576X/1.

\bigskip
\bigskip
\noindent
\footnotesize {\bf Authors' addresses:}

\smallskip
\noindent
Erik Burman,  \quad \hfill \addressuclshort\\
{\tt e.burman@ucl.ac.uk}

\smallskip
\noindent
Peter Hansbo,  \quad \hfill \addressjushort\\
{\tt peter.hansbo@ju.se}

\smallskip
\noindent
Mats G. Larson,  \quad \hfill \addressumushort\\
{\tt mats.larson@umu.se}

\smallskip
\noindent
Karl Larsson,  \quad \hfill \addressumushort\\
{\tt karl.larsson@umu.se}

\end{document}